\documentclass[10pt]{amsart}
\usepackage[utf8]{inputenc}

\usepackage{amssymb,amsthm,amsmath}
\usepackage{mathrsfs}
\usepackage{enumerate}
\usepackage{graphicx,xcolor}
\usepackage{setspace}
\usepackage[hidelinks]{hyperref}

\usepackage{comment}

\usepackage[paper=a4paper, left=1.2in, right=1.2in, top=1.2in, bottom=1.2in]{geometry}
\newcommand{\ls}{\leqslant}
\newcommand{\gr}{\geqslant}
\newcommand{\E}{\mathbb{E}}

\renewcommand{\P}{\mathbb{P}}

\newcommand{\R}{\mathbb{R}}

\newcommand{\e}{\varepsilon}

\newcommand{\red}{\color{red}}

\renewcommand{\mathfrak}{\mathcal}

\DeclareMathOperator{\Var}{Var}

\DeclareMathOperator{\vol}{vol}
\DeclareMathOperator{\conv}{conv}

\usepackage{dsfont}
\usepackage{upgreek}

\usepackage[symbol]{footmisc}

\newtheorem{theorem}{Theorem}[section]
\newtheorem{lemma}[theorem]{Lemma}
\newtheorem{corollary}[theorem]{Corollary}
\newtheorem{proposition}[theorem]{Proposition}

\theoremstyle{remark}
\newtheorem{remark}[theorem]{Remark}
\newtheorem{conjecture}{Conjecture}

\theoremstyle{definition}
\newtheorem{defn}[theorem]{Definition}

\title{Sharp estimates for the Cram\'{e}r transform of log-concave measures and geometric applications}

\author{Silouanos Brazitikos}

\author{Giorgos Chasapis}

\address{Department of Mathematics \& Applied Mathematics, University of Crete, Voutes Campus, 70013 Heraklion, Greece}
\email{silouanb@uoc.gr}

\address{Department of Mathematics, University of Ioannina, University Campus, 45110 Ioannina, Greece}
\email{gchasapis@uoi.gr}

\thanks{}

\date{June 13, 2025.}

\begin{document}

\begin{abstract} 
We establish a new comparison between the Legendre transform of the cumulant generating function and the half-space depth of an arbitrary log-concave probability distribution on the real line, that carries on to the multidimensional setting. Combined with sharp estimates for the Cram\'{e}r transform of rotationally invariant measures, we are led to some new phase-transition type results for the asymptotics of the expected measure of random polytopes. As a byproduct of our analysis, we address a question on the sharp exponential separability constant for log-concave distributions, in the symmetric case.
\end{abstract}

\maketitle

\bigskip

\begin{footnotesize}

\end{footnotesize}

\bigskip

\section{Introduction}
In the present work we are concerned with a study of the Cram\'{e}r transform of log-concave probability distributions on $\mathbb{R}^n$, particularly in connection with the so-called notion of half-space depth and aiming towards an improved understanding of asymptotic geometric properties of random structures in high-dimensional spaces.

Given a probability measure $\mu $ on ${\mathbb R}^n$, Tukey's half-space depth is defined
for any $x\in\mathbb {R}^n$ by
\[
q_{\mu }(x)=\inf\{\mu (H):H\in {\mathcal H}(x)\},
\]
where ${\mathcal H}(x)$ is the set of all closed half-spaces $H$ of ${\mathbb R}^n$ containing $x$. Such a functional can be considered as a measure of centrality or outlyingness for multivariate data points, quantifying the ``depth'' of a point $x$ with respect to a set of data by considering the fraction of half-spaces containing $x$. The first work in statistics where some form of the half-space depth appears
is an article of Hodges \cite{Hodges-1955} from 1955. Tukey introduced the half-space depth for data sets in \cite{Tukey-1975} as a
tool that enables efficient visualization of random samples in the plane. The term ``depth" also comes
from Tukey's article. A formal definition of the half-space depth was given by
Donoho and Gasko in \cite{Donoho-Gasko-1982} (see also \cite{Small-1987}). Quite expectedly, depth functionals emerge also in a more geometric context as measures of symmetry or moduli for the approximation of convex bodies by polytopes. We refer the reader to the survey article of Nagy, Sch\"{u}tt and Werner
\cite{Nagy-Schutt-Werner-2019} for an extensive treatise on data depth functionals, their connections to convex geometry and more, as well as many references.

By the Cram\'{e}r transform of a probability distribution $\mu$ on $\mathbb{R}^n$ on the other hand, we refer to the functional defined by
\[
\Lambda^\ast_\mu(x) = \sup_{\xi\in\mathbb{R}^n}\left(\langle x,\xi\rangle-\log\E_\mu e^{\langle X,\xi\rangle}\right),
\]
that is, the Legendre transform of the cumulant generating function (or Log-Laplace transform) of $\mu$. The terminology stems from the pivotal role that this quantity plays in Cram\'{e}r's theorem within the theory of Large Deviations \cite{Dembo-Zeitouni}. In the context of Asymptotic Geometric Analysis, the cumulant generating function arises naturally as a central object in the study of high-dimensional convex bodies and log-concave measures (see for example \cite{Klartag-perturb, Eldan-Klartag, Klartag-Milman}). Interestingly, in a framework closer to mathematical optimisation, Bubeck and Eldan \cite{Bubeck-Eldan} (see also \cite{Chewi}) proved that the Cram\'{e}r transform of the uniform probability distribution on a convex body in $\mathbb{R}^n$ is an $n$-self-concordant barrier, giving the first universal barrier for convex bodies with optimal self-concordance parameter (for more on barrier functions, interior point methods and the theory of convex optimisation, we point the interested reader to \cite{Nesterov-Nemirovski}).

These two quantities are intimately connected; since
\[
q_\mu(x) = \inf_{\xi\in S^{n-1}}\P_\mu(\langle X,\xi\rangle\gr \langle x,\xi\rangle),
\]
we have, using Markov's inequality, that
\[
q_\mu(x)\ls \exp(-(\langle x,\xi\rangle-\log\E_\mu e^{\langle X,\xi\rangle}))
\]
for every $\xi\in S^{n-1}$. Taking infimum over $\xi$, it follows that the inequality $q_\mu(x)\ls e^{-\Lambda_\mu^\ast(x)}$, or equivalenty,
\begin{equation}\label{eq:q-L*-chernoff}
\Lambda_\mu^\ast(x)\ls \log\frac{1}{q_\mu(x)}
\end{equation}
holds for every $x$ and any probability measure $\mu$ on $\mathbb{R}^n$. Exploring the precision of this upper estimate has been crucial in many instances for problems regarding the geometry of high-dimensional random sets. It was probably first noted by Dyer, F\"{u}redi and McDiarmid in \cite{DFM} that, for sufficiently large $n$, \eqref{eq:q-L*-chernoff} can be essentially reversed when one takes the measure $\mu$ to be the uniform probability distribution on the vertices (or interior) of the $n$-dimensional cube $[-1,1]^n$, an observation that came to be crucial also in the work of B\'{a}r\'{a}ny and P\'{o}r \cite{Barany-Por} on the expected facet number of 0/1 polytopes. Subsequently, Gatzouras and Giannopoulos \cite{Gatz-Gian} (see also \cite{Gatz-Gian-LDP}, as well as \cite{GGM1}, \cite{GGM2}) carefully analysed the argument of \cite{DFM} to obtain a similar result for more general even product probability measures $\mu_n=\mu^{\otimes n}$ on $\mathbb{R}^n$ with compact support that satisfy a certain regularity assumption, namely that
\begin{equation}\label{L_[limit]}
\lim _{x \to x^\ast} \frac{-\ln \mu([x, \infty))}{\Lambda_\mu^\ast(x)}=1,
\end{equation}
where $x^\ast$ is the right endpoint of the support of $\mu$ (some instances of the same behaviour for product measures with no compact support were recently studied in \cite{Pafis}).

Quite recently, a reasonable reversal of \eqref{eq:q-L*-chernoff} was established in \cite{BGP} in a fairly general setting; the inequality
\begin{equation}\label{eq:L*-lower_BGP}
\Lambda^\ast_\mu(x)\gr \log\frac{1}{q_\mu(x)}-5\sqrt{n}
\end{equation}
is true for every $x\in\mathbb{R}^n$ if $\mu$ is the uniform probability distribution on any centered convex body of volume 1 in $\mathbb{R}^n$ (as is explained in \cite{BGP}, the $O(\sqrt{n})$-error term is fairly good for the sake of the applications considered by the authors). One of our main results is the following lower estimate on $\Lambda_\mu^\ast$ for any log-concave probability measure $\mu$ on $\mathbb{R}^n$.
\begin{theorem}\label{thm:L*-q-lower-general}
Let $\mu$ be a log-concave probability on $\mathbb{R}^n$. Then for any $\varepsilon\in(0,1)$,
\begin{equation}\label{eq:L*-q-lower-general}
\Lambda_\mu^\ast(x)\gr (1-\varepsilon)\log\frac{1}{q_\mu(x)}+\log\frac{\varepsilon}{2^{1-\varepsilon}}
\end{equation}
holds for any $x\in\mathbb{R}^n$.
\end{theorem}
Optimising over $\varepsilon>0$, we can see that the right hand side of \eqref{eq:L*-q-lower-general} is maximised when $\varepsilon=(-\log(2q_\mu(x)))^{-1}$. This leads to the lower estimate
\[
\Lambda^\ast_\mu(x)\gr \log\frac{1}{2q_\mu(x)\log\frac{1}{2q_\mu(x)}}-1,
\]
for every $x$ with $q_\mu(x)<1/(2e)$.

A key ingredient for the proof of Theorem \ref{thm:L*-q-lower-general} is a sharp inequality for the negative moments of $q_\mu(X)$, for any real random variable $X$ with distribution $\mu$ (see Proposition \ref{prop:q-moments}). Then, fundamental convexity considerations lead us to the 1-dimensional version of \eqref{eq:L*-q-lower-general}, which can be extended to higher dimensions thanks to the log-concavity of marginal distributions of $\mu$.

A similar argument to that used for the proof of Theorem \ref{thm:L*-q-lower-general} lets us provide a characterisation of the measures $\mu$ on the real line that satisfy \eqref{L_[limit]}, under no symmetry or compactness assumption.
\begin{theorem}\label{thm:L*-cond-char}
Let $\mu$ be probability measure on $\mathbb{R}$ and let $x^\ast\in(-\infty,+\infty]$ denote the right endpoint of $\mathrm{supp}(\mu)$. Then, 
\[
\lim_{x\to x^\ast} \frac{-\log\mu([x,\infty))}{\Lambda_\mu^\ast(x)}=1
\]
holds if and only if there exists a convex function $V(x)$ such that $\lim_{x\to x^\ast}\frac{-\log\mu([x,\infty))}{V(x)}=1$.
\end{theorem}

Next, we focus on radially symmetric probability distributions on $\mathbb{R}^n$. We determine the radially symmetric distributions that maximise/minimise the Cram\'{e}r transform functional pointwise in the log-concave setting under an isotropicity assumption, and establish sharp bounds for the order of growth of $\E\Lambda^\ast_\mu$ and $\E e^{-\Lambda^\ast_\mu}$ for certain classes of rotationally invariant log-concave probability measures. Our methods combine an analysis of the cumulant generating function of the marginals of the spherical distribution, majorisation-type arguments and log-concavity considerations that let in for a fruitful comparison of the underlying densities. One particular outcome of our results is the following.
\begin{theorem}\label{thm:exp-one-shot-symmetr}
Let $X=(X_1,\ldots,X_n)$ be a random vector in $\mathbb{R}^n$ with independent and symmetric log-concave components. Then $\E e^{-\Lambda_X^\ast}\ls c_{\mathrm{exp}}^n$ for $c_{\mathrm{exp}}=\E e^{-\Lambda_Y^\ast}$, where $Y$ is a standard double exponential random variable, i.e. with density $\frac{1}{2}e^{-\frac{|x|}{2}}$.
\end{theorem}
Theorem \ref{thm:exp-one-shot-symmetr} provides a partial positive answer, in the symmetric case, to a question raised in \cite{GGT-2023} on the exponential one-shot separability constant of log-concave distributions. In connection with stochastic separability and applications to machine learning  and
error-correction mechanisms in artificial intelligence systems, the authors in \cite{GGT-2023} considered the quantity $\E q_\mu$ as a ``measure of separability'' of the distribution $\mu$, and have pursued the understanding of cases of distributions for which $\E q_\mu$ decreases exponentially fast with the dimension (the definition of ``one-shot separability'' is stated in \cite{GGT-2023} only for product distributions, but one could as well consider the general case). It is suggested in the same work that the constant $c_{\mathrm{exp}}$ above provides the sharp upper bound in general, for log-concave product distributions. We actually conjecture that the result of Theorem \ref{thm:exp-one-shot-symmetr} provides the correct upper bound even in the high-dimensional case (for not-necessarily product measures), at least in the rotationally-invariant setting. Relative to the case of product distributions, the general upper bound $1-2\cdot 10^{-5}$ for $\E e^{-\Lambda^\ast_\mu}$, where $\mu$ is any log-concave probability measure on $\mathbb{R}$, has been announced in \cite[Proposition 3]{GGT-2023}. However, the referenced manuscript has not yet come to our attention. Theorem \ref{thm:L*-q-lower-general} lets us also provide an improved, although still suboptimal, general estimate.
\begin{proposition}\label{prop:exp-seper-R}
For every log-concave random variable $X$ we have $\E e^{-\Lambda_X^\ast(X)}\ls 1-(4e)^{-1}$.
\end{proposition}
The proof of Proposition \ref{prop:exp-seper-R} is provided in Section \ref{sec:thresholds-prod}.

The starting point and main motivation for the results in this work have been questions on the geometry of high-dimensional random sets, in particular phase transition-type phenomena for the asymptotics of the expected measure of randomly generated convex hulls in $\mathbb{R}^n$. Following a long line of research \cite{DFM, Gatz-Gian, Pivovarov, BCGTT, Pafis} that we briefly recall in the next section, as well as the recent treatment of the problem in the general setting of log-concave probability measures \cite{BGP}, we are led to a natural conjecture that we are now able to support in some new cases, including that of general product probability measures on $\mathbb{R}^n$ and several rotationally invariant distributions in high-dimensions.

The rest of the paper is organised as follows: In Section \ref{sec:preliminaries} we have gathered some preliminaries and further develop the background related to our considered applications. The proofs of Theorems \ref{thm:L*-q-lower-general} and \ref{thm:L*-cond-char} are presented in Section \ref{sec:L*-low}. Results regarding the Cr\'{a}mer transform of rotationally invariant distributions are collected in Section \ref{sec:rot-inv}. Finally, in Section \ref{sec:thresholds} we record some implications of our results in various settings for the expected measure-threshold problem.

Throughout the text we will use standard big-O notation. In particular, $f(x)=\omega(g(x))$ stands for $\lim_{x\to \infty}f(x)/g(x)=+\infty$ and $f(x)\sim g(x)$ stands for $\lim_{x\to\infty}f(x)/g(x)=1$.

\section{Preliminaries and Background}\label{sec:preliminaries}

\subsection{Log-concave probability measures}

We denote by $\langle\cdot,\cdot\rangle$ the standard Euclidean inner product in $\mathbb{R}^n$ and by $|\cdot|$ the Euclidean norm in the respective dimension. The $n$-dimensional closed Euclidean unit ball is denoted by $B_2^n$ and $S^{n-1}$ is its boundary, the unit sphere. It is known that $S^{n-1}$ is equipped with a unique rotationally invariant probability measure, which we usually denote by $\sigma$ (surpassing explicit reference to the dimension $n$). The $n$-dimensional Lebesgue measure of a measurable set $A\subset \mathbb{R}^n$ will be denoted by $\vol_n(A)$.

We say that a subset $K$ of $\mathbb{R}^n$ is a convex body if it is compact and convex, with a nonempty interior. A convex body $K$ is (centrally) symmetric if $K=-K$ and centered if its barycenter ${\rm bar}(K)=\vol_n(K)^{-1}\int_K x\,dx$ is the origin. A Borel measure $\mu$ on $\mathbb{R}^n$ is called log-concave if $\mu((1-\lambda)A+\lambda B)\gr \mu(A)^{1-\lambda}\mu(B)^\lambda$ for any pair of measurable $A,B\subset \mathbb{R}^n$ and $\lambda\in(0,1)$. We say that a random vector $X$ in $\mathbb{R}^n$ is log-concave if it is distributed according to a log-concave probability measure $\mu$ on $\mathbb{R}^n$. It is known (see \cite{Borell-1974}) that any log-concave probability measure $\mu$ on $\mathbb{R}^n$ which is not concentrated on a hyperplane (that is, $\mu(H)<1$ for every hyperplane $H$) has a log-concave density $f_\mu$. We will mainly focus on this non-degenerate setting. A standard example of a log-concave probability measure on $\mathbb{R}^n$ is that of the uniform measure on a convex body $K$ in $\mathbb{R}^n$, which we might denote by $\mu_K$. Clearly then $f_{\mu_K}=\vol_n(K)^{-1}\mathds{1}_K$ (which is log-concave, by the convexity of $K$). Another common example is the $n$-dimensional Gaussian measure, which we denote by $\gamma_n$ and has density $f(x)=(2\pi)^{-\frac{n}{2}}e^{-\frac{|x|^2}{2}}$.

We say that a convex body $K$ in $\mathbb{R}^n$ is isotropic if it is centered, $\vol_n(K)=1$, and there is a constant $L_K>0$ such that
\[
\left(\int_K\langle x,\theta\rangle^2\,dx\right)^{\frac{1}{2}}=L_K
\]
for every $\theta\in S^{n-1}$. More generally, we say that a log-concave probability measure $\mu$ on $\mathbb{R}^n$ is isotropic if
\[
\int_{\mathbb{R}^n}\langle x,\theta\rangle\,d\mu(x)=0 \qquad\hbox{and}\qquad \int_{\mathbb{R}^n}\langle x,\theta\rangle^2\,d\mu(x)=1,
\]
both hold for all $\theta\in S^{n-1}$. Note that, under this normalisation, a convex body $K$ in $\mathbb{R}^n$ is isotropic if and only if the probability measure with density $f(x)=L_K^n\mathds{1}_{L_K^{-1}K}$ is isotropic. It is known that every log-concave measure $\mu$ on $\mathbb{R}^n$ can be pushed forward to an isotropic measure by an affine transformation. The isotropic constant of a log-concave probability measure on $\mathbb{R}^n$ with density $f_\mu$ is defined by
\[
L_\mu:=\left(\mathrm{det}(\mathrm{Cov}(\mu))\|f_\mu\|_\infty^2\right)^{\frac{1}{2n}},
\]
where $\mathrm{Cov}(\mu)$ is the covariance matrix of $\mu$. Note that if $\mu$ is isotropic then $\mathrm{Cov}(\mu)$ is the identity matrix. The celebrated isotropic constant conjecture (equivalent to the hyperplane conjecture posed by Bourgain \cite{Bourgain-1991}) asking if there exists an absolute constant $C>0$ such that $L_n:=\sup L_\mu\ls C$ for every $n\in\mathbb{N}$ (where the supremum is taken over all isotropic log-concave probability measures on $\mathbb{R}^n$), had been a central open problem for more than 30 years in the field of Asymptotic Geometric Analysis. It was not only until very recently (while the present paper was undergoing the peer review process) that this question was settled to the affirmative by Klartag-Lehec \cite{Klartag-Lehec}, following a crucial observation by Guan \cite{Guan} after several other recent breakthrough advances. We refer the interested reader to \cite{BGVV} for details on the rich history of this problem, its connection with other major open questions in the field, as well as an extensive study of the theory of log-concave measures.

\subsection{Threshold for the expected measure of random sets}

For any dimension $n\in\mathbb{N}$, let $N>n$ and $X_1,\ldots,X_N$ be independent random vectors in $\mathbb{R}^n$ with a common distribution $\mu_n$. We denote by $K_N:=\conv\{X_1,\ldots,X_N\}$, the convex hull of $X_1,\ldots,X_N$. Our main object of study is the behaviour of the sequence $a_n=\E \mu_n(K_N)$ when $n$ tends to infinity. Here, expectation is taken with repsect to the product distribution $\mu_n^{\otimes N}$ (in a more general setting, one could even take the points $X_i$ distributed according to different measures $\mu_n^{i}$). Note that the quantity $\E\mu_n(K_N)$ is invariant under affine transformations of $\mu_n$, so we may in general assume that $\mu_n$ is centered.

We will see that usually the asymptotics of $a_n$ will depend on the magnitude of the number of random points $N$ chosen, with respect to the dimension $n$. To make the notion of a ``sharp threshold'' precise, we rely on the following definitions, essentially first considered in \cite{BGP}. Given $\delta\in(0,1)$ and a probability measure $\mu$ on $\mathbb{R}^n$, we let
\[
\varrho_1(\mu,\delta):= \sup\{r>0: \E\mu(K_N)\ls \delta, \hbox{ for every } N\ls \exp(r)\},
\]
and
\[
\varrho_2(\mu,\delta):= \inf\{r>0: \E\mu(K_N)\gr 1-\delta, \hbox{ for every } N\gr \exp(r)\}.
\]
Let also
\[
\varrho(\mu,\delta) := \varrho_2(\mu,\delta)-\varrho_1(\mu,\delta).
\]
With this notation at hand, we give the following definition.
\begin{defn}\label{def:sharp-threshold}
Let $(\mu_n)_{n\in\mathbb{N}}$ be a sequence of probability measures $\mu_n$ on $\mathbb{R}^n$, and $(\mathfrak{T}_n)_n$ a sequence of positive numbers. We say that $(\mu_n)_{n\in\mathbb{N}}$ exhibits a sharp threshold around $\mathfrak{T}_n$ if for every $\varepsilon\in(0,1)$ there is $n_0=n_0(\varepsilon)\in\mathbb{N}$ such that for every $n\gr n_0$,
\[
\varrho_1(\mu_n,\delta_n)\gr(1-\varepsilon)\mathfrak{T}_n \qquad\hbox{ and }\qquad \varrho_2(\mu_n,\delta_n)\ls(1+\varepsilon)\mathfrak{T}_n,
\]
for some sequence $(\delta_n)$ with $\lim_{n\to\infty}\delta_n=0$.
\end{defn}
Clearly, if $(\mu_n)_{n\in\mathbb{N}}$ exhibits a sharp threshold around $\mathfrak{T}_n$ under the above definition, then it is the case that for every fixed $\varepsilon\in(0,1)$, any choice $N\ls e^{(1-\varepsilon)\mathfrak{T}_n}$ will imply that $\E\mu_n(K_N)\to 0$ as $n\to\infty$ and respectively $\E\mu_n(K_N)\to 1$ as $n\to\infty$ if $N\gr e^{(1+\varepsilon)\mathfrak{T}_n}$. That is, the expected measure of $K_N$ exhibits an asymptotic phase transition from 0 to 1 around $N=e^{\mathfrak{T}_n}$, with a ``threshold window'' of the order of $\mathfrak{T}_n^{-1}\varrho(\mu_n,\delta_n)\ls 2\varepsilon$.

Using this terminology, the authors in \cite{BGP} put into a common perspective a number of related older results. They also studied extensively the quantities $\varrho_1$ and $\varrho_2$ under a log-concavity assumption. In the following subsection we elaborate further on the history of the problem and recall several achievements and natural conjectures.

\subsection{Half-space depth and the Cram\'{e}r transform}

Recall the notion of Tukey's half-space depth defined for any probability measure $\mu$ and point $x$ in $\mathbb{R}^n$ by
\[
q_{\mu }(x)=\inf\{\mu (H):H \hbox{ is a closed half-space containing } x\}.
\]
In the sequel, given a random vector $X$ distributed according to $\mu$ we might use either the notation $q_\mu$ or $q_X$ for the half-space depth of $\mu$ (or even relax this to $q$, if the underlying distribution is obvious from the context). By definition, in dimension 1 we have $q_\mu(x)=\min\{\mu((-\infty,x]),\mu([x,\infty))\}$ for every $x\in\mathbb{R}$. In general, in any dimension $n$, note that since any half-space $H$ that contains $x$ is a set of the form $\{y:\langle y-x,\xi\rangle \gr 0\}$ for some $\xi\in S^{n-1}$, $q_\mu(x)$ can be realised as
\[
q_\mu(x) = \inf_{\xi\in S^{n-1}}\P_\mu(\langle X,\xi\rangle\gr \langle x,\xi\rangle).
\]
It is then easy to see that $q_Y(ax+b)=q_X(x)$ for any $a,b\in\mathbb{R}$ whenever $Y$ is equidistributed to $aX+b$ and, as a consequence, if $\mu$ is the distribution of $X$ and $\nu$ is the distribution of $aX+b$ then
\[
\E_{\nu}\Phi(q_{\nu}(Y)) = \E_\mu\Phi(q_{\nu}(aX+b)) = \E_\mu\Phi(q_\mu(X))
\]
holds for any function $\Phi:(0,\infty)\to(0,\infty)$, that is, the quantity $\E_\mu\Phi(q_\mu)$ is invariant under affine transformations of $\mu$. Another important property that we will rely on is the fact that $q_\mu$ respects log-concavity.
\begin{lemma}\label{lem:q-logconc}
If $\mu$ is a probability measure on $\mathbb{R}^n$ whose marginals have log-concave tails, then the function $q_\mu$ is log-concave. In particular if $\mu$ has a log-concave density in $\mathbb{R}^n$, then $q_\mu$ is log-concave.
\end{lemma}
\begin{proof}
Assume that $X$ is distributed according to $\mu$. Fix some arbitrary $\xi\in S^{n-1}$. By our hypothesis, for any $x,y\in\mathbb{R}^n$ and $\lambda\in(0,1)$
\[
\P(\langle X,\xi\rangle\gr\langle (1-\lambda)x+\lambda y,\xi\rangle)\gr\left(\P(\langle X,\xi\rangle\gr\langle x,\xi\rangle)\right)^{1-\lambda}\left(\P(\langle X,\xi\rangle\gr\langle y,\xi\rangle)\right)^\lambda\gr q_\mu(x)^{1-\lambda}q_\mu(y)^\lambda.
\]
Taking infimum over $\xi\in S^{n-1}$ concludes the proof.
\end{proof}
In a seminal work, Dyer, F\"{u}redi and McDiarmid  \cite{DFM} were the first to show that a threshold-type behaviour like that described in Definition \ref{def:sharp-threshold} holds for the expected volume of random polytopes with independent vertices uniformly distributed in the unit cube $[-1,1]^n$; if $X_1,\ldots,X_N$ are as such and $K_N=\conv\{X_1,\ldots,X_N\}$, then there is an absolute constant $\kappa>0$ such that for every $\varepsilon\in(0,1)$,
\begin{equation}\label{eq:DFM-cube-threshold}
\lim_{n\to\infty} 2^{-n}\E\vol_n(K_N) = \begin{cases} 1, &\hbox{ if } N\gr (\kappa+\varepsilon)^n\\ 0, &\hbox{ if } N\ls (\kappa-\varepsilon)^n \end{cases}
\end{equation}
(a similar result holds also in the case that the random points are chosen uniformly from the vertices of the cube). 
The proof in \cite{DFM} relies on an ingenious blend of probabilistic and geometric arguments, manifested in the following central lemma.
\begin{lemma}\label{lem:DFM}
Let $K_N$ be the convex hull of $N$ random vectors independently chosen with respect to a common probability distribution $\mu$ on $\mathbb{R}^n$. Then for every Borel subset $A$ of $\mathbb{R}^n$,
\begin{equation}\label{eq:DFM-lemma-upper}
\E\mu(K_N)\ls \mu(A)+N\cdot \sup_{x\in A^c}q_\mu(x).
\end{equation}
If additionally $\mu$ assigns zero mass to every hyperplane in $\mathbb{R}^n$, then
\begin{equation}\label{eq:DFM-lemma-lower}
\E\mu(K_N)\gr \mu(A)\left(1-2\binom{N}{n}\left(1-\inf_{x\in A}q_\mu(x)\right)^{N-n}\right).
\end{equation}
\end{lemma}
It is noteworthy that the proof of Lemma \ref{lem:DFM} does not rely on particular properties of the underlying distribution $\mu$, turning this into a versatile tool that has been the starting point in most, if not all, subsequent works establishing similar threshold-type phenomena.

The exact value of the constant $\kappa$ in \eqref{eq:DFM-cube-threshold} is $\kappa=2\pi e^{-\gamma-\frac{1}{2}}$, where $\gamma$ is Euler's constant. One can see that this is exactly $\E\Lambda^\ast_U(U)$, where $U$ is uniform in $[-1,1]$. Let us recall some basic facts on the log-Laplace and Cram\'{e}r transforms of log-concave probability distributions. We denote by $\Lambda_\mu$ the cumulant generating function of a probability measure $\mu$ on $\mathbb{R}^n$, that is
\[
\Lambda_\mu(\xi)=\log\E e^{\langle X,\xi\rangle},
\]
for every $\xi \in\mathbb{R}^n$, where $X$ is distributed according to $\mu$ (again, in this case we will feel free to interchange between the notation $\Lambda_\mu$ and $\Lambda_X$). Note that $\Lambda_\mu$ is a convex function on $\mathbb{R}^n$ (this is a direct consequence of H\"{o}lder's inequality), which is equal to 0 at the origin. Moreover, $\Lambda_\mu\gr 0$ holds if $\mu$ is centered, by Jensen's inequality. The Cram\'{e}r transform of $\mu$ is defined as the Legendre--Fenchel transform of $\Lambda_\mu$, that is,
\begin{equation}\label{eq:L*-def}
\Lambda^\ast_\mu(x) = \sup_{\xi\in\mathbb{R}^n} \left(\langle x,\xi\rangle-\Lambda_\mu(\xi)\right).
\end{equation}
As a Legendre transform, $\Lambda_\mu^\ast$ is a non-negative convex and lower semicontinuous function on $\mathbb{R}^n$. In the case of one-dimensional symmetric distributions, which will be of interest to us, $\Lambda_\mu$ and $\Lambda_\mu^\ast$ are even functions, which are well defined when $\Lambda_\mu$ is finite on some neighborhood of 0. The set $\{\Lambda_\mu<\infty\}$ is in that case a symmetric open interval $(-t^\ast,t^\ast)$ for some $t^\ast\in(0,\infty]$, and $\Lambda_\mu$ is $C^\infty$ on $(-t^\ast,t^\ast)$. If moreover $\mu$ is log-concave and $\mathrm{supp}(\mu)=(-x^\ast,x^\ast)$, $x^\ast\in(0,\infty]$, then $\Lambda_\mu$ is differentiable, with $\Lambda'_\mu$ strictly increasing on $(-t^\ast,t^\ast)$ and onto $(-x^\ast,x^\ast)$, in particular
\[
\lim_{t\to \pm t^\ast}\Lambda_\mu'(t)=\pm x^\ast
\]
(see \cite[Lemma 2.3]{Pafis} for a proof of these claims). The Legendre transform  $\Lambda^\ast_\mu$ on the other hand is $C^\infty$ on $\mathrm{supp}(\mu)$ and satisfies $({\Lambda_\mu^\ast})'=(\Lambda_\mu')^{-1}$. Moreover, for every $x\in(-x^\ast,x^\ast)$,
\[
\Lambda^\ast_\mu(x) = x\cdot t_x-\Lambda_\mu(t_x),
\]
where $t_x=({\Lambda^\ast_\mu})'(x)$ (equivalently, $x=\Lambda_\mu'(t)$). We point the reader to the monograph \cite{Rockafellar} for more on the general theory of Legendre duality as well as \cite{Dembo-Zeitouni, Deuschel-Stroock} for connections of the Cram\'{e}r transform to the theory of Large Deviations.

Note finally that if $X$ is a random vector with distribution $\mu$ in $\mathbb{R}^n$, and $\nu$ is the distribution of $Y=aX+b$, for any real numbers $a,b$, it is straightforward to check that $\Lambda^\ast_\nu(x)=\Lambda^\ast_\mu\left(\frac{x-b}{a}\right)$ for every $x\in\mathbb{R}^n$. It follows then that for any function $\Phi:(0,\infty)\to(0,\infty)$,
\[
\E_\nu(\Phi(\Lambda_\nu^\ast(Y)))=\E_\mu(\Phi(\Lambda_\nu^\ast(aX+b))) = \E_X\Phi(\Lambda_\mu^\ast(X)),
\]
in other words $\E_\mu(\Phi(\Lambda^\ast_\mu))$ is affinely invariant.

From the very definitions, it is easy to verify that for any probability measure $\mu$, the pointwise inequality
\begin{equation}\label{eq:q-L^ast-upper}
q_\mu(x)\ls e^{-\Lambda^\ast_\mu(x)}
\end{equation}
holds (this direct consequence of the Markov/Chernoff bound was already explained in the Introduction). This fundamental observation led the authors of \cite{DFM} to the ``small N'' part of \eqref{eq:DFM-cube-threshold} as follows: If $x$ is such that $q_\mu(x)>e^{-\alpha}$ for some $\alpha>0$, then $\Lambda_\mu^\ast(x)<\alpha$. Independence however implies that if $\mu$ is the uniform probability on the cube, then $\Lambda_\mu^\ast(x)=\sum_{j=1}^n\Lambda^\ast_U(x_j)$, where $U$ is uniform in $[-1,1]$, so a choice of $\alpha=(1-\varepsilon/2)n\kappa$, for $\kappa=\E\Lambda^\ast_U$ and $A=\{x\in\mathbb{R}^n:q(x)>\alpha\}$ in Lemma \ref{lem:DFM} implies
\[
\mu(A)\ls \P\left(\frac{1}{n}\sum_{j=1}^n\Lambda^\ast_U(U_j)<(1-\varepsilon/2)\E\Lambda_U^\ast\right),
\]
where $U_1,\ldots,U_n$ are independent copies of $U$. The latter probability decreases to 0 with $n$, by the Central Limit Theorem, and this will eventually lead us to the upper bound in \eqref{eq:DFM-cube-threshold}, thanks to \eqref{eq:DFM-lemma-upper}. For the other half of the Theorem, a comparison reverse to that of \eqref{eq:q-L^ast-upper} would let us argue in a similar fashion when $N$ is ``large''. Second to Lemma \ref{lem:DFM}, the non-trivial part  in the proof of \eqref{eq:DFM-cube-threshold} was to show that for any $\varepsilon>0$ the inequality $q_\mu(x)\gr e^{-\Lambda_\mu^\ast(x)-\varepsilon}$ also holds when $n$ is sufficiently large.

The argument of \cite{DFM} was carefully reviewed by Gatzouras and Giannopoulos \cite{Gatz-Gian}, who verified that \eqref{eq:DFM-cube-threshold} still holds for any even and compactly supported product probability measure $\mu_n=\mu^{\otimes n}$ on $\mathbb{R}^n$, provided that $\E\Lambda^\ast_\mu<\infty$ and $\mu$ satisfies the ``$\Lambda^\ast$-condition'',
\begin{equation}\tag{$*$}\label{eq:L^ast-condition}
\lim_{x\to x^\ast}\frac{-\log\mu([x,\infty))}{\Lambda^\ast_\mu(x)}=1,
\end{equation}
where $\mathrm{supp}(\mu)=[-x^\ast,x^\ast]$. Condition \eqref{eq:L^ast-condition} essentially requires that $q_\mu$ and $e^{-\Lambda_\mu^\ast}$ tend to decrease at the same rate near the boundary. The authors in \cite{Gatz-Gian} also provided sufficient conditions for a measure to be ``admissible'' in the above sense. The result of Gatzouras--Giannopoulos has been quite recently generalised further by Pafis \cite{Pafis}, who showed that \eqref{eq:L^ast-condition} implies a sharp threshold for $\E\mu^{\otimes n}(K_N)$ around $n\E\Lambda_\mu$ even in cases of several even product measures with non-compact support, an important example demonstrated by densities of the form $c_pe^{-|x|^p}$, $p\gr 1$. After Theorem \ref{thm:L*-q-lower-general}, we can verify that the same behaviour is exhibited by arbitrary log-concave probability measures and moreover determine the sharpness of the threshold window $\varrho$.
\begin{theorem}\label{thm:prod-threshold}
Let $(\lambda_m)_{m\in\mathbb{N}}$ be a sequence of log-concave probability measures on $\mathbb{R}$ and for every $n\in\mathbb{N}$ set $\mu_n=\lambda_1\otimes\ldots\otimes \lambda_n$. For any $\varepsilon\in(0,1)$,
\[
\lim_{n\to\infty}\E\mu_n(K_N) = \begin{cases} 0, &\hbox{ if } N\ls \exp((1-\varepsilon)\E\Lambda_{\mu_n}^\ast),\\ 1, &\hbox{ if } N\gr \exp((1+\varepsilon)\E\Lambda_{\mu_n}^\ast) \end{cases}.
\]
Moreover, $(\E\Lambda^\ast_{\mu_n})^{-1}\varrho(\mu_n,\delta)\ls C\max\{(\delta n)^{-\frac{1}{2}},\log n/n\}$ for every $\delta\in(0,1)$, where $C>0$ is an absolute constant.
\end{theorem}

As it turns out, it is considerably harder for one to successfully treat the case of arbitrary, non-product, distributions on $\mathbb{R}^n$. Some partial results have been obtained for the case of the uniform measure on the simplex \cite{FPT}, and under general log-concavity or $s$-concavity assumptions \cite{CTV}. Pivovarov \cite{Pivovarov} utilised rotational invariance to obtain a sharp phase transition for the expected \emph{volume} of $K_N$ in the case that the vertices are distributed independently according to the uniform measure on $S^{n-1}$ or the Gaussian measure on $\mathbb{R}^n$. His results were later extended in \cite{BCGTT} to the case of general Beta and Beta-prime distributions (see also \cite{BKT} for an intriguing complementary result). In all these examples, rotational invariance stands up somehow to the lack of independence, reducing again the problem to one-dimensional considerations. Still however it is not clear to us whether a unified argument could let one handle the case of arbitrary rotationally invariant, say log-concave, probability distributions on $\mathbb{R}^n$. Our results allow us to give a partial answer under some further convexity assumptions. In the sequel, we identify any function $f:\mathbb{R}^n\to(0,\infty)$ such that $f(x)=f(|x|)$ with the real function $\tilde{f}:(0,\infty)\to(0,\infty)$ given by $\tilde{f}(r)=f(re_1)$.
\begin{theorem}\label{thm:LC-conv-threshold}
Suppose that $\mu_n$ is a rotationally invariant log-concave probability distribution on $\mathbb{R}^n$ with a density $f_n$ such that $x\mapsto f_n(\sqrt{|x|})$ is a log-convex function. Then $\{\mu_n\}_{n\in\mathbb{N}}$ exhibits a sharp threshold around $\E\Lambda^\ast_{\mu_n}$, with  $\lim_{n\to\infty}(\E\Lambda_{\mu_n}^\ast)^{-1}\varrho({\mu_n},\delta)=0$, for any $\delta=\omega(n^{-1})$.
\end{theorem}
Another example of rotationally invariant distributions on $\mathbb{R}^n$ that we will deal with is that of Beta-distributions $\mu_{n,\beta}$, that is densities $f_{n,\beta}$ proportional to $(1-|x|^2)^\beta$, $\beta>-1$, supported on $B_2^n$. The name is due to the fact that if $X$ is a random vector in $\mathbb{R}^n$ with density $f_{n,\beta}$, then $|X|^2$ follows the Beta distribution with parameters $n/2$ and $\beta+1$ on $(0,1)$. We also remark that when $\beta=0$, $\mu_{n,0}$ is nothing other than the uniform distribution on $B_2^n$, whereas the uniform probability $\sigma$ on $S^{n-1}$ is recovered as the weak limit of $(\mu_{n,\beta})_{\beta}$ with $\beta\to 1$. Early works where the beta distribution comes up, primarily in the context of polytopal approximation of convex bodies include \cite{Eddy-Gale}, \cite{Buchta-Muller}, \cite{Buchta-Muller-Tichy}, \cite{Affentranger}, \cite{Dwyer}. Geometric properties of beta-polytopes have been since then a standard object of study, more recently also under the asymptotic viewpoint; apart from the already mentioned works \cite{BCGTT} and \cite{BKT}, we point the reader to \cite{BGTTTW}, \cite{Chasapis-Skarmogiannis}, \cite{Kabluchko-Temesvari-Thaele}, \cite{Kabluchko-DCG}, \cite{Kabluchko-Advances}, this list being of course aything but extensive.

Note that the case of Beta-distributions is not covered by Theorem \ref{thm:LC-conv-threshold} (in fact, $x\mapsto f_{n,\beta}(\sqrt{|x|})$ is log-concave). We can however show that a result similar to that of \cite{BCGTT} still holds for the Beta-measure in the place of volume.
\begin{theorem}\label{thm:beta-threshold}
Let $-1<\beta\ls cn$ for some absolute constant $c>0$, and $\mu_{n,\beta}$ be the probability measure on $\mathbb{R}^n$ with density $f_{n,\beta}$. Then for every $\varepsilon>0$,
\[
\lim_{n\to\infty} \E\mu_{n,\beta}(K_N) = \begin{cases} 1, &\hbox{ if } N\gr \exp((1+\varepsilon)\E\Lambda_{\mu_{n,\beta}}^\ast),\\ 0, &\hbox{ if } N\ls \exp((1-\varepsilon)\E\Lambda_{\mu_{n,\beta}}^\ast).\end{cases}
\]
Moreover, $\E\Lambda_{\mu_{n,\beta}}^\ast \sim \left(\beta+\frac{n+1}{2}\right)\log n$ and $\varrho(\mu_{n,\beta},\delta)= O((\delta\log^2 n)^{-1/2})$, for any $\delta\in(0,1)$. 
\end{theorem}

In view of all past and recent results mentioned in this section, we are led to believe that in general, in the log-concave setting, there is a natural candidate for the ``threshold constant'' for the expected measure of random convex sets. We are yet to find a counterexample to the following.
\begin{conjecture}\label{conj:threshold}
For any log-concave probability measure $\mu_n$ on $\mathbb{R}^n$, the family $\{\mu_n\}_{n\in\mathbb{N}}$ exhibits a sharp threshold in the sense of Definition \ref{def:sharp-threshold} around $\mathfrak{T}_n=\E\Lambda^\ast_{\mu_n}$.
\end{conjecture}
Even though Conjecture \ref{conj:threshold} still remains elusive even for the case of uniform probability densities on convex bodies, our current results provide some further examples of measures that support it, namely $n$-products of arbitrary one-dimensional log-concave distributions (Theorem \ref{thm:prod-threshold}) and a number of rotationally invariant probability distributions on $\mathbb{R}^n$ including densities proportional to $e^{{-|x|}^p}$, $p\in(0,2)$, Gamma-type densities proportional to $|x|^{\alpha-1}e^{-|x|}$, $0<\alpha<n$ (Theorem \ref{thm:LC-conv-threshold}, see Remark \ref{rem:rot-inv}) and Beta-type densities (Theorem \ref{thm:beta-threshold}).

Recently, a systematic study of the problem was carried out in \cite{BGP}, under only the assumption of log-concavity of the underlying distribution. The strategy developed in \cite{BGP} builds upon the previous work of the same authors \cite{BGP-q}, where very accurate upper and lower bounds are provided for the expected value of the half-space depth $q_\mu$ for an arbitrary log-concave probability measure $\mu$ on $\mathbb{R}^n$. In particular, it was proved in \cite{BGP-q} that there are absolute constants $c_1,c_2>0$ such that 
\begin{equation}\label{eq:Eq-BGP}
e^{-c_1n}\ls \E_\mu q_\mu(X)\ls e^{-c_2n/L_\mu^2}
\end{equation}
for any log-concave probability measure $\mu$ on $\mathbb{R}^n$, where $L_\mu$ is the isotropic constant of $\mu$. We remark that the upper bound in \eqref{eq:Eq-BGP} follows from the stronger inequality (cf. \eqref{eq:q-L*-chernoff})
\begin{equation}\label{eq:Ee^-L*-BGP}
\E_\mu e^{-\Lambda^\ast_\mu(X)}\ls e^{-c_2n/L_\mu^2},
\end{equation}
for any log-concave probability measure $\mu$ on $\mathbb{R}^n$. Note also that, by Jensen's inequality, the latter implies that
\begin{equation}\label{eq:EL*-BGP}
\E_\mu\Lambda_\mu^\ast(X)\gr cn/L_\mu^2.
\end{equation}
Given the recent resolution of the isotropic constant conjecture, inequality \eqref{eq:EL*-BGP} essentially states that on average, the Cram\'{e}r transform of any log-concave probability measure on $\mathbb{R}^n$ grows at least linearly with the dimension. This is a very useful estimate that we will recall in a number of instances in the sequel. In Section \ref{sec:rot-inv}, we will sharpen the bounds \eqref{eq:Ee^-L*-BGP} and \eqref{eq:EL*-BGP} for some classes of rotationally invariant distributions.

The problem of obtaining sharp upper bounds for $\E_\mu\Lambda_\mu^\ast$ is also interesting and relevant to our applications. Let us mention here that very recently Giannopoulos and Tziotziou \cite{GianTz}, building upon the previous works \cite{BGP-q}, \cite{BGP}, proved that the inequality
\begin{equation}\label{eq:EL*-upper-general}
\left(\E_\mu(\Lambda_\mu^\ast)^2\right)^{\frac{1}{2}}\ls cn\log n
\end{equation}
for some absolute constant $c>0$, is valid for any centered log-concave probability measure on $\mathbb{R}^n$. This estimate is optimal; one can check that $\E\Lambda^\ast_\mu$ grows at the same rate with $n$ in the case that $\mu=\mu_{B_2^n}$ is the uniform measure on the Euclidean ball.

The approach of \cite{BGP} led to a unified treatment (and refinement) of the results in \cite{DFM} and \cite{Pivovarov}, still leaving however several open questions for the general case. Theorem \ref{thm:L*-q-lower-general} actually lets us extend some of the central results of \cite{BGP} in the general setting of log-concave measures. This together with the study of the rotationally-invariant setting carried out in Section \ref{sec:rot-inv} will lead us to the results announced in the present section. We elaborate more on the details and conclude with the proofs in Section \ref{sec:thresholds}.

\section{Lower estimate for the Cram\'{e}r transform under log-concavity}\label{sec:L*-low}

Given two random variables $X$ and $Y$, we say that $X$ stochastically dominates $Y$ (in the first order) if $\P(Y\ls x)\ls \P(X\ls x)$ for every $x\in\mathbb{R}$. It is well known that this is equivalent to the condition $\E(f(X))\ls\E(f(Y))$ for any non-decreasing $f:\mathbb{R}\to\mathbb{R}$. Given a random variable $X$, we denote by $F(x)=\P(X\ls x)$, the cumulative distribution function of $X$, and $Z(x)=\P(X\gr x)$. It is well known that if $X$ has a continuous c.d.f. $F$, then the random variable $F(X)$ is uniform on $[0,1]$. This can be verified immediately in the case that $F$ is strictly increasing, since then
\[
\P(F(X)\ls x) = \P(X\ls F^{-1}(x)) = x = \P(U\ls x)
\]
for any $0<x<1$, where $U$ is uniform on $[0,1]$ (the argument needs to be only slightly modified to capture the general case). In the next lemma, we verify that $F(X)$ and $Z(X)$ always stochastically dominate the uniform distribution.
\begin{lemma}\label{lem:stoc-dom}
Let $U$ be a random variable uniformly distributed on $[0,1]$, and $X$ be any random variable. Then the random variables $F(X), Z(X)$ stochastically dominate $U$.
\end{lemma}
\begin{proof}
We need to show that for any $s\in(0,1)$,
\[
\max\{\P(Z(X)\ls s),\P(F(X)\ls s)\} \ls \P(U\ls s)=s.
\]
Since $F$ is right-continuous and non-decreasing, it has countably many jump discontinuities, say $(x_k)_{k\in\mathbb{Z}}$, with $x_{k-1}<x_k$ for every $k\in\mathbb{Z}$. For each $k\in\mathbb{N}$, let $a_k=\lim_{x\to x_k^-}F(x)$ and $b_k=F(x_k)$, so that we always have $a_k<b_k$. Note also that $\P(F(X)\in[a_k,b_k))=0$ for all $k$. Given $s\in(0,1)$, we distinguish two cases: If $s\in[a_k,b_k)$ for some $k\in\mathbb{N}$, then
\[
\P(F(X)\ls s)=\P(F(X)<a_k)+\P(F(X)\in[a_k,s]) = \P(X< x_k)=a_k<s.
\]
If again $b_{k-1}<a_k$ and $s\in[b_{k-1},a_k)$, then $F$ is continuous on $s$. Letting $m:=\sup\{y:F(y)\ls s\}$, we can check that
\[
\P(F(X)\ls s)=\P(X\in\{y:F(y)\ls s\}) = \P(X\ls m)=F(m)=s,
\]
by continuity.

The result for $Z$ then follows: since $F(X)$ stochastically dominates $U$ for any $X$, $Z(X)=F(-X)$ also stochastically dominates $U$.
\end{proof}

We will use Lemma \ref{lem:stoc-dom} to obtain good bounds on the $p$-th moments of $q_X(X)$, $p>-1$, for an arbitrary random variable $X$.
\begin{proposition}\label{prop:q-moments}
Let $X$ be any random variable and $p> -1$.\\
\noindent {\rm a)} If $p>0$, then $(2^p(p+1))^{-1}\ls \E(q(X)^p)\ls 1$.\\
\noindent {\rm b)} If $p<0$, then $1\ls \E(q(X)^p)\ls (2^p(p+1))^{-1}$.
\end{proposition}
\begin{remark}\label{rem:q-uniform} Recall that if $X$ is a continuous random variable, then $q(x)=\min\{F(x),1-F(x)\}$. If moreover $F$ is continuous, then $F(X)$ is uniform on $[0,1]$, so for any $p> -1$,
\begin{align*}
\E(q(X)^p) &=\E(\min\{F(X),1-F(X)\}^p) = \E(\min\{U,1-U\}^p)\\ &= \int_0^\frac{1}{2} u^p\,du + \int_\frac{1}{2}^1 (1-u)^p\,du = (2^p(p+1))^{-1}.
\end{align*}
This, together with the trivial fact $q=1$ for any atomic distribution, shows that the inequalities in Proposition \ref{prop:q-moments} are sharp.
\end{remark}

\begin{proof} The upper and lower bound in (a) and (b) respectively are due to the trivial inequality $q(x)\ls 1$. To prove the lower bound in (a), let $M=2^{-p}$ and write
\[
\E(q(X)^p)=\int_0^\infty \P(q(X)^p\gr x)\,dx\gr \int_0^M \P(q(X)^p\gr x)\,dx \gr \int_0^M 1-\P(q(X)\ls x^{\frac{1}{p}})\,dx.
\]
By Lemma \ref{lem:stoc-dom},
\[
\P(q(X)\ls x^{\frac{1}{p}})\ls \P(F(X)\ls x^{\frac{1}{p}})+\P(Z(X)\ls x^{\frac{1}{p}})\ls 2x^{\frac{1}{p}}.
\]
It follows then that
\[
\E(q(X)^p) \gr \int_0^M 1-2x^{\frac{1}{p}}\,dx = M-2\frac{p}{p+1}M^{\frac{p+1}{p}}=\frac{2^{-p}}{p+1}.
\]
Similarly, in the case that $-1<p<0$, we have that
\[
\E(q(X)^p)\ls \int_0^M 1\,dx+\int_M^\infty \P(q(X)^p\gr x) = M+\int_M^\infty \P(q(X)\ls x^{\frac{1}{p}})\,dx,
\]
where again we let $M=2^{-p}$, and using $\P(q(X)\ls x^{\frac{1}{p}})\ls 2x^{\frac{1}{p}}$ this time we get
\[
\E(q(X)^p)\ls M-2\frac{p}{p+1}M^{\frac{p}{p+1}}=\frac{2^{-p}}{p+1},
\]
which is the upper bound in (b).
\end{proof}

With the aid of Proposition \ref{prop:q-moments} we are in place to prove the following one-dimensional inequality from which Theorem \ref{thm:L*-q-lower-general} will follow. In the statement we consider also the case of discrete measures in $\mathbb{R}$. We say that a discrete probability distribution $\mu$ with probability mass function $p:\mathbb{Z}\to(0,\infty)$ is log-concave if $p$ is log-concave on $\mathbb{Z}$, that is, $p(k)^2\gr p(k-1)p(k+1)$ for any $k\in\mathbb{Z}$. It can be proved that the tail and distribution functions of a log-concave p.m.f. are also log-concave in the above sense (see for example \cite[Theorem 2]{Pinelis} for a more general statement), and consequently this is also the case for $q_\mu$. For an arbitrary subset $A$ of $\mathbb{R}$ we let here $\conv(A)$ denote the smallest interval containing $A$.
\begin{theorem}\label{thm:L*-condition}
Let $\mu$ be a log-concave probability measure on $\mathbb{R}$ that is either absolutely continuous with respect to the Lebesgue measure, or supported on a subset of the integers. For every $\varepsilon\in(0,1)$ and every $x\in\conv(\mathrm{supp}(\mu))$, 
\begin{equation}\label{eq:L*-lower}
\Lambda_\mu^\ast(x)\geq (1-\varepsilon)\log\frac{1}{q_{\mu}(x)}+\log\frac{\varepsilon}{2^{1-\varepsilon}}.
\end{equation}
\end{theorem}
\begin{proof}
To start with the absolutely continuous case, fix $\varepsilon\in(0,1)$ and $x\in\conv(\mathrm{supp}(\mu))$. Since $\mu$ is log-concave, the functions $\mu((-\infty,y])$ and $\mu([y,\infty))$ are both log-concave and so is  $q_{\mu}$ (recall Lemma \ref{lem:q-logconc}). Consider the function $g(y)=-(1-\varepsilon)\log q_{\mu}(y)$. By the convexity of $g$, there are $t,b\in\mathbb{R}$ such that the tangent line $l(y)=ty+b$ of $g$ at $x$ satisfies $l(x)=g(x)$ and $l(y)\ls g(y)$ for all $y\neq x$. It follows that $\E e^{l(X)}\ls \E e^{g(X)}$, which is equivalent to $-\log\E e^{g(X)}\ls -\log\E e^{l(X)}=-b-\Lambda(t)$. Since $l(x)=g(x)$, we eventually get that
\[
g(x)-\log\E e^{g(X)} \ls xt-\Lambda_\mu(t)\ls \Lambda_{\mu}^\ast(x).
\]
This fact combined with Proposition \ref{prop:q-moments} (b) implies the desired inequality.

The argument is almost identical in the discrete setting: If the p.m.f. of $\mu$ is log-concave on $\mathbb{Z}$ then both $F$ and $Z$ are log-concave on $\mathbb{Z}$ and consequently so is $q_\mu$. Now let $g$ be the linear extension of $f:=-(1-\varepsilon)\log q_\mu$ on $\mathbb{R}$ (that is, $g(k+x)=f(k)+x(f(k+1) -f(k))$ for every $k\in\mathrm{supp}(\mu)$ and $x\in [0,1]$). It is easy to see that $g$ is a convex function on $\mathbb{R}$ ($g$ is actually well-defined on the smallest interval containing $\mathrm{supp}\mu$) which is piecewise linear, so it stays above its ``tangents''. Moreover, $\E e^{g(X)}=\E q(X)^{\varepsilon-1}$, since $g=f$ on $\mathrm{supp}(\mu)$, which lets us conclude the proof applying again the bound in Proposition \ref{prop:q-moments} (b).
\end{proof}

\begin{remark}
As is clear from the proof, equality in \eqref{eq:L*-lower} can be attained for some $x,\varepsilon$ if $\mu$ has a log-affine density (take for example $x=2\log 2$ and $\varepsilon=(\log 2)^{-1}$ when $\mu$ is the standard exponential distribution).
\end{remark}

Following a similar reasoning, we can also prove the characterisation of the condition \eqref{L_[limit]}.

\begin{theorem}\label{thm:L^ast-cond} Let $\mu$ be a probability measure on $\mathbb{R}$ and let $x^\ast\in[0,+\infty]$ denote the right endpoint of $\mathrm{supp}(\mu)$. Then, 
\[
\lim_{x\to x^\ast} \frac{-\log\mu([x,\infty))}{\Lambda_\mu^\ast(x)}=1.
\]
if and only if there exists a convex function $V(x)$, such that $\lim_{x\to x^\ast}\frac{-\log\mu([x,\infty))}{V(x)}=1$.
\end{theorem}
\begin{proof}
The one direction is trivial, since if the limit is equal to $1$, we take $V=\Lambda_\mu^\ast$. For the other direction, it is easy to check that $-\log\mu([x,\infty))\gr \Lambda_\mu^\ast(x)$ holds true for all $x$
(again, this is nothing else than \eqref{eq:q-L*-chernoff}). It is then enough to prove that
\[
\lim_{x\to x^\ast}\frac{\Lambda_\mu^\ast(x)}{-\log\mu([x,\infty))}\gr 1-\varepsilon, \qquad\hbox{ for all } \varepsilon>0.
\]
To this end, fix $\varepsilon>0$ and let $\delta=\varepsilon/(4-\varepsilon)$. Then there is an $M\in(0,x^\ast)$, such that 
\begin{equation}\label{sandw}(1-\delta)V(x)\leq-\log\mu([x,\infty))\leq (1+\delta)V(x),\end{equation} for all $M<x<x^\ast$.
Fix $x\in(M,x^\ast)$. Since $W=(1-\delta)^2V$ is convex, there are $t,b\in\mathbb{R}$ such that the tangent line $l(y)=ty+b$ of $W$ at $x$ satisfies $l(x)=W(x)$ and $l(y)\ls W(y)$ for all $y\neq x$. It follows that $\E e^{l(X)}\ls \E e^{W(X)}$, which is equivalent to $-\log\E e^{W(X)}\ls -\log\E e^{l(X)}=-b- \Lambda_\mu(t)$. Since $l(x)=W(x)$, we eventually get that
\[
W(x)-\log\E e^{W(X)} \ls xt- \Lambda_\mu(t)\ls \Lambda_{\mu}^\ast(x).
\]

Using \eqref{sandw} we get that 
$$\frac{(1-\delta)^2}{1+\delta}(-\log\mu([x,\infty))-\log\E e^{-(1-\delta)\log(\mu([x,\infty))}\ls W(x)-\log\E e^{W(X)}$$
Finally, the computation in the proof of Proposition \ref{prop:q-moments}(b) yields
$$-\frac{(1-\delta)^2}{1+\delta}\log\mu\left([x,\infty)\right)+\log\frac{\delta}{2^{-\delta}}\leq \Lambda_{\mu}^\ast(x),$$ for all $x>M$. Since $\frac{(1-\delta)^2}{1+\delta}\gr(1-\varepsilon/2)^2\gr 1-\varepsilon$, it follows that
\[
\frac{\Lambda_\mu^\ast(x)}{-\log(\mu([x,\infty)))} \gr 1-\varepsilon+\frac{\log (2^{-\delta}/\delta)}{\log(\mu([x,\infty)))}
\]
Now we have to distinguish between two cases: if $\P(X=x^\ast)=0$, where $X$ is distributed according to $\mu$, then taking the limit as $x\to x^\ast$ makes the right hand side above equal to $1-\varepsilon$ (since then $\lim_{x\to x^\ast}\mu([x,\infty))=0$), and since $\varepsilon>0$ was arbitrary the wanted claim is proved.

In the case $\P(X=x^\ast)>0$ on the other hand, the statement of the Theorem has already been handled in \cite[Lemma 2.8]{Gatz-Gian}.
\end{proof}

\begin{remark}
Notice that Theorem \ref{thm:L^ast-cond} gives us new, even non log-concave, examples of probability measures on $\mathbb{R}$ for which the $\Lambda^\ast$-condition \eqref{eq:L^ast-condition} holds. Consider for example
\[
\mu([x,\infty))=pe^{-\lambda_1x}+(1-p)e^{-\lambda_2x} \qquad x\in\mathbb{R},
\]
for some $p\in(0,1)$ and $0<\lambda_1<\lambda_2$. Then $\mu([x,\infty))$ is not log-concave, however taking $V(x)=\lambda_1x$ we get
\[
\lim_{x\to+\infty}\frac{-\log(\mu([x,\infty)))}{V(x)}=1.
\]
\end{remark}

We finally complete the proof of Theorem \ref{thm:L*-q-lower-general} exploiting the connection between $\Lambda_X^\ast$, $q_X$ for a random vector $X$ in $\mathbb{R}^n$ and the respective functionals of the one-dimensional marginal distributions of $X$.

\begin{corollary}\label{cor:q-lower}
Let $\mu$ be probability measure on $\mathbb{R}^n$ whose one-dimensional marginals are log-concave. Then for every $\eta>0$,
\[
\Lambda_\mu^\ast(x)\gr (1-\eta)\log\left(\frac{1}{q_\mu(x)}\right)+\log\frac{\eta}{2^{1-\eta}}.
\]
\end{corollary}
\begin{proof}
If $X$ is a random vector in $\mathbb{R}^n$ distributed according to $\mu$, for any fixed $\xi\in S^{n-1}$ we consider the random variable $\xi_X:=\langle X,\xi\rangle$. By assumption, the distribution of $\xi_X$ is log-concave so Theorem \ref{thm:L*-condition} implies that
\[
\Lambda_{\xi_X}^\ast(\langle x,\xi\rangle)\gr (1-\eta)\log\left(\frac{1}{q_{\xi_X}(\langle x,\xi\rangle)}\right)+\log\frac{\eta}{2^{1-\eta}},
\]
for every $x\in\mathbb{R}^n$. Next, note that for every $x\in\mathbb{R}^n$,
\begin{align*}
\Lambda_X^\ast(x) &= \sup_{y\in\mathbb{R}^n}\left(\langle x,y\rangle-\Lambda_X(y)\right) = \sup_{(\xi,t)\in S^{n-1}\times\mathbb{R}} \left(t\langle x,\xi\rangle-\Lambda_{\langle X,\xi\rangle}(t)\right) = \sup_{\xi\in S^{n-1}}\Lambda^\ast_{\xi_X}(\langle x,\xi\rangle).
\end{align*}
Putting things together, we get that
\begin{align*}
\Lambda_X^\ast(x) &\gr (1-\eta)\sup_{\xi\in S^{n-1}}\log\left(\frac{1}{q_{\xi_X}(\langle x,\xi\rangle)}\right)+\log\frac{\eta}{2^{1-\eta}} =(1-\eta)\log\left(\frac{1}{q_X(x)}\right)+\log\frac{\eta}{2^{1-\eta}},
\end{align*}
since
\begin{align*}
\sup_{\xi\in S^{n-1}}\log\left(\frac{1}{q_{\xi_X}(\langle x,\xi\rangle)}\right) &= \log\left(\frac{1}{\inf_{\xi\in S^{n-1}}q_{\xi_X}(\langle x,\xi\rangle)}\right)\\
        &= \log\left(\frac{1}{\inf_{\xi\in S^{n-1}}\P(\langle X,\xi\rangle\gr \langle x,\xi\rangle)}\right) = \log\left(\frac{1}{q_X(x)}\right).\qedhere
\end{align*}
\end{proof}

\begin{remark}
The assumption of log-concavity in Theorem \ref{thm:L*-condition} can be actually relaxed to the assumption that the functions $x\mapsto\mu((-\infty,x])$ and $x\mapsto\mu([x,\infty))$ are both log-concave. If $\mu$ is assumed to be an even measure, it is enough that $\mu$ has log-concave tails. Similarly, if the measure $\mu$ in Corollary \ref{cor:q-lower} is assumed to be even, it is enough to assume that one of the above functions is log-concave.
\end{remark}

\begin{remark}
The inequality of Corollary \ref{cor:q-lower} can serve as an alternate to \eqref{eq:L*-lower_BGP}, which is however valid for arbitrary log-concave measures on $\mathbb{R}^n$. To compare the two estimates, note that e.g. a choice $\eta=1/n$ in Corollary \ref{cor:q-lower} gives the lower bound $\frac{n-1}{n}(\log(1/q(x))-\log 2)-\log n$, which is larger than $\log(1/q(x))-5\sqrt{n}$ if $\log(1/q(x))\ls 5n^{3/2}$. This is practically enough for our applications, since we know that on average $\log(1/q(x))$ is of the order of $n$ (cf. \eqref{eq:Eq-BGP}).
\end{remark}

\section{Sharp bounds for rotationally invariant distributions}\label{sec:rot-inv}

If $X$ is a random vector in $\mathbb{R}^n$ following a rotationally invariant distribution, then $X\overset{d}{=}R\cdot\vartheta$, where $R$ is the radial distribution of $X$, $\vartheta$ is uniform on the Euclidean unit sphere $S^{n-1}$ and $R,\vartheta$ are independent. One can check that $R$ is a nonnegative random variable which is log-concave whenever $X$ is log-concave (if $f(x)=f(|x|)$ is the density of $X$, then the density of $R$ is $f_R(r)=n\omega_nr^{n-1}f(r)$\footnote[2]{Everywhere in this section, we identify any function $f:\mathbb{R}^n\to(0,\infty)$ such that $f(x)=f(|x|)$ with the real function $\tilde{f}:(0,\infty)\to(0,\infty)$ given by $\tilde{f}(r)=f(re_1)$.}). It is also straightforward to check that for every $x\in\mathbb{R}^n$, $\Lambda_X(x)=\Lambda_{X_1}(|x|)$ and, consequently, $\Lambda_X^\ast(x)=\Lambda_{X_1}^\ast(|x|)$, where $X_1=\langle X,e_1\rangle$. When $X$ is isotropic it follows that $\E R^2=n$.
We will denote everywhere in this section by $\vartheta$ a random vector uniformly distributed on $S^{n-1}$. Moreover, $G$ stands for a standard Gaussian random vector in $\mathbb{R}^n$ and $\mathcal E$ is the random vector with density $f(x)=c_ne^{-\sqrt{n+1}|x|}$, where $c_n=\frac{(n+1)^{\frac{n}{2}}}{n\omega_n\Gamma(n)}$, the normalisation chosen so that $\mathcal E$ is isotropic.

\subsection{Pointwise bounds on the Cram\'{e}r transform} Our first goal is to provide upper and lower pointwise bounds for the Cram\'{e}r transform of an arbitrary rotationally invariant distribution on $\mathbb{R}^n$, under an isotropicity assumption. In the case of the uniform distribution on the sphere, note that
\begin{equation}\label{eq:Lambda_theta1}
e^{\Lambda_{\vartheta_1}(t)} = \int_{S^{n-1}}e^{t\theta_1}\,d\sigma(\theta) = \frac{\Gamma\left(\frac{n}{2}\right)}{\sqrt{\pi}\Gamma\left(\frac{n-1}{2}\right)}\int_{-1}^1 e^{tr}(1-r^2)^{\frac{n-3}{2}}\,dr = \Gamma\left(\frac{n}{2}\right)\left(\frac{t}{2}\right)^{-\frac{n-2}{2}}I_{\frac{n-2}{2}}(t),
\end{equation}
where $I_a$ stands for the modified Bessel function of the first kind of order $a$. Using the series expansion
\begin{equation}\label{eq:bessel-series}
\left(\frac{t}{2}\right)^{-a}I_a(t) = \sum_{k=0}^\infty \frac{1}{k!\Gamma(k+a+1)}\left(\frac{t}{2}\right)^{2k},
\end{equation}
we get the following expression for the Laplace transform of an arbitrary radially symmetric random vector in terms of the even moments of its radial distribution.
\begin{lemma}\label{lem:L_sigma}
Let $X$ be a rotationally invariant random vector in $\mathbb{R}^n$. If $R_X$, the radial distribution of $X$, has finite moments of all orders then
\[
e^{\Lambda_{X_1}(t)} = \Gamma\left(\frac{n}{2}\right)\sum_{k=0}^\infty\frac{1}{k!\Gamma(k+n/2)}\left(\frac{t}{2}\right)^{2k}\E(R_X^{2k}).
\]
\end{lemma}
\begin{proof}
Since $X\overset{d}{=}R_X\vartheta$, and due to the independence of $R$ and $\vartheta$, we have $e^{\Lambda_{X_1}(t)} = \E e^{tX_1} = \E e^{tR_X\vartheta_1} = \E_{R_X} e^{\Lambda_{\vartheta_1}(tR_X)}$. The wanted identity follows immediately then by \eqref{eq:Lambda_theta1} combined with the series representation \eqref{eq:bessel-series}.
\end{proof}
As a consequence of Lemma \ref{lem:L_sigma} and Jensen's inequality for the function $x\mapsto x^{2k}$, we get the following general bounds.
\begin{proposition}\label{prop:L*-general-upper}
Let $X$ be a rotationally invariant random vector in $\mathbb{R}^n$. If $X$ is isotropic, then $\Lambda_X(\xi)\gr\Lambda_{\sqrt{n}\vartheta}(\xi)$ for every $\xi\in\mathbb{R}^n$, and, as a consequence, $\Lambda_X^\ast\ls \Lambda^\ast_{\sqrt{n}\vartheta}\,$.
\end{proposition}
We next note that among all log-concave and radially symmetric distributions, the exponential is the one that stands on the other extreme.
\begin{proposition}\label{prop:Exp-min}
If $X$ is an isotropic and log-concave random vector in $\mathbb{R}^n$ with a radially symmetric density, then $\Lambda_{X}\ls \Lambda_{\mathcal E}$ and, as a consequence, $\Lambda_{X}^\ast\gr\Lambda_{\mathcal E}^\ast$.
\end{proposition}
For the proof of Proposition \ref{prop:Exp-min}, it is enough to verify that $\Lambda_{X_1}\ls\Lambda_{{\mathcal E}_1}$ (recall that $\Lambda_X(x)=\Lambda_{X_1}(|x|)$). This is a direct consequence of the following classical fact that goes back to Borell \cite{Borell-1974} (see also \cite[Theorem 2.2.5]{BGVV}).
\begin{lemma}\label{lem:Borell}
Given $f:\mathbb{R}_+\to\mathbb{R}_+$, we define
\[
\Psi_f(p) = \frac{\int_0^\infty r^pf(r)\,dr}{\Gamma(p+1)}.
\]
If $f$ is log-concave, then $\Psi_f$ is also log-concave on $[0,\infty)$.
\end{lemma}
\begin{proof}[Proof of Proposition \ref{prop:Exp-min}]
We will use again Lemma \ref{lem:L_sigma} combined with the fact that $X$ is equidistributed to a product $R\cdot\vartheta$ with $\vartheta$ uniform on $S^{n-1}$ and $R$ independent of $\vartheta$, with density $n\omega_nr^{n-1}f(r)$. In the notation of Lemma \ref{lem:Borell}, note that $n\omega_n\Psi_f(n-1)=\frac{1}{\Gamma(n)}$ and $n\omega_n\Psi_f(n+1) = \frac{\E R^2}{\Gamma(n+2)} = \frac{n}{\Gamma(n+2)}$, since $X$ is isotropic. For any $k\in\mathbb{N}$, log-concavity of $f$, and hence of $\Psi_f$, implies that
\[
\Psi_f(n+1)\gr \Psi_f(2k+n-1)^\frac{1}{k}\Psi_f(n-1)^{1-\frac{1}{k}},
\]
which can be equivalently recast as
\[
\E R^{2k} \ls \Gamma(2k+n)(n\omega_n\Psi_f(n+1))^k(n\omega_n\Psi_f(n-1))^{1-k} = \frac{\Gamma(2k+n)}{\Gamma(n)}\left(\frac{n\Gamma(n)}{\Gamma(n+2)}\right)^k = \E R_{\mathcal E}^{2k},
\]
where $R_{\mathcal E}$ stands for the radial distribution of ${\mathcal E}$. The last equality is due to the fact that
\[
\E R_{\mathcal E}^{2k} = \frac{(n+1)^{\frac{n}{2}}}{\Gamma(n)}\int_0^\infty r^{n+2k-1}e^{-\sqrt{n+1}\,r}\,dr = \frac{\Gamma(2k+n)}{\Gamma(n)}(n+1)^{-k}.
\]
Lemma \ref{lem:L_sigma} then concludes the proof.
\end{proof}
\begin{remark}\label{rem:exp-L-L*-comput}
The computation of the moments $\E R_{\mathcal E}^{2k}$ carried out in the above proof lets us also explicitly compute the Laplace transform $\Lambda_{\mathcal E_1}$ of the one-dimensional marginal distribution of $\mathcal E$. Using Lemma \ref{lem:L_sigma} we arrive at
\[
e^{\Lambda_{{\mathcal E}_1}(t)} = \frac{\Gamma\left(\frac{n}{2}\right)}{\Gamma(n)}\sum_{k=0}^\infty \frac{\Gamma(2k+n)}{k!\Gamma\left(k+\frac{n}{2}\right)}\left(\frac{t}{2\sqrt{n+1}}\right)^{2k} = \left(1-\frac{t^2}{n+1}\right)^{-\frac{n+1}{2}},
\]
that is
\begin{equation}\label{eq:Exp-L}
\Lambda_{\mathcal E_1}(t) = -\frac{n+1}{2}\log\left(1-\frac{t^2}{n+1}\right).
\end{equation}
By a straightforward computation then we can check that the Cram\'{e}r transform of $\mathcal E_1$ is equal to
\begin{equation}\label{eq:Exp-L*}
\Lambda_{\mathcal E_1}^\ast(x) = \frac{n+1}{2}\left(\sqrt{\frac{4x^2}{n+1}+1}-1-\log\frac{\sqrt{\frac{4x^2}{n+1}+1}+1}{2}\right).
\end{equation}
\end{remark}

We can provide some improved bounds for the Laplace and Cram\'{e}r transforms of radially symmetric log-concave distributions under certain further convexity assumptions. We consider the following classes of functions
\[
\mathcal{LC}_{conc}:= \{f:(0,\infty)\to\mathbb{R}\,|\,\hbox{ the function } x\mapsto f(\sqrt{x}) \hbox{ is log-concave}\}
\]
and
\[
\mathcal{LC}_{conv}:= \{f:(0,\infty)\to\mathbb{R}\,|\,\hbox{ the function } x\mapsto f(\sqrt{x}) \hbox{ is log-convex}\}.
\]
Note that a Gaussian function $g(x)=e^{-x^2}$ belongs to both of these classes, since $g(\sqrt{x})$ is log-linear. We will provide Gaussian bounds for the moment generating function of a rotationally invariant and log-concave random vector $X$ if its density lies either in $\mathcal{LC}_{conv}$ or $\mathcal{LC}_{conc}$. First, we remark the implications that this additional assumption bears on the convexity properties of the moment generating function. The following fact is a consequence of a result from \cite{Barthe-Kold}.
\begin{lemma}\label{lem:Barthe-Kold-conv}
Let $X$ be a random vector in $\mathbb{R}^n$ which is distributed according to a radially symmetric density $f_X(x)=f_X(|x|)$. If $f_X\in\mathcal{LC}_{conc}$ (resp. $\mathcal{LC}_{conv}$) then $t\mapsto \Lambda_{X_1}(\sqrt{t})$ is concave (resp. convex).
\end{lemma}
\begin{proof}
Denote by $f$ the one-dimensional marginal of $f_X$ and assume first that $f_X\in\mathcal{LC}_{conc}$. The statement of the Lemma will follow immediately by \cite[Theorem 12]{Barthe-Kold} once we show that $f\in\mathcal{LC}_{conc}$. Note that
\[
f(t)=\int_{\mathbb{R}^{n-1}}f_X((t,y))\,dy = \int_{\mathbb{R}^{n-1}}f_X(\sqrt{t^2+|y|^2})\,dy,
\]
so that
\[
f(\sqrt{(1-\lambda)t+\lambda s}) = \int_{\mathbb{R}^{n-1}}f_X(\sqrt{(1-\lambda)t+\lambda s+|y|^2})\,dy
\]
for any $\lambda\in(0,1)$ and $t,s>0$. Define the functions $m,g,h:\mathbb{R}^{n-1}\to (0,\infty)$ by $m(x)=f_X(\sqrt{(1-\lambda)t+\lambda s+|x|^2})$, $g(x)=f_X(\sqrt{t+|x|^2})$ and $h(x)=f_X(\sqrt{s+|x|^2})$ and note that, since $f_X(|x|)$ is a decreasing function of $|x|$, it follows by the convexity of $z\mapsto |z|^2$ that
\begin{align*}
m((1-\lambda)x+\lambda y) &\gr f_X(\sqrt{(1-\lambda)(t+|x|^2)+\lambda(s+|y|^2)})\\
                          &\gr f_X(\sqrt{t+|x|^2})^{1-\lambda}f_X(\sqrt{s+|y|^2})^\lambda = g(x)^{1-\lambda}h(y)^\lambda,
\end{align*}
where we have also used the assumption $f_X\in\mathcal{LC}_{conc}$. It follows then by the Pr\'{e}kopa-Leindler inequality that $f(\sqrt{(1-\lambda)t+\lambda s})\gr f(\sqrt{t})^{1-\lambda}f(\sqrt{s})^\lambda$.

For the case $f_X\in\mathcal{LC}_{conv}$ we can simply argue that
\begin{align*}
f(\sqrt{(1-\lambda)t+\lambda s}) &= \int_{\mathbb{R}^{n-1}}f_X(\sqrt{(1-\lambda)t+\lambda s+|y|^2})\,dy\\
                                 &= \int_{\mathbb{R}^{n-1}}f_X(\sqrt{(1-\lambda)(t+|y|^2)+\lambda (s+|y|^2}))\,dy\\
                                 &\ls \int_{\mathbb{R}^{n-1}}f_X(\sqrt{t+|y|^2})^{1-\lambda}f_X(\sqrt{s+|y|^2})^\lambda\,dy\\
                                 &\ls f(\sqrt{t})^{1-\lambda}f(\sqrt{s})^\lambda,
\end{align*}
by H\"{o}lder's inequality.
\end{proof}
The next fact is straightforward to verify.
\begin{lemma}\label{lem:L-sqrt-convex}
If the function $t\mapsto \Lambda_{X_1}(\sqrt{t})$ is convex (resp. concave) on $(0,+\infty)$ then the function $x\mapsto \Lambda_{X_1}^\ast(\sqrt{x})$ is concave (resp. convex) on $(0,\infty)$.
\end{lemma}

\begin{proof}
We sketch the argument for the case that $t\mapsto \Lambda_{X_1}(\sqrt{t})$ is convex (for the concave case, only the ensuing inequalities are reversed). If $f:(0,\infty)\to(0,\infty)$ defined by $f(t)=\Lambda_{X_1}(\sqrt{t})$ is convex, then
\[
0\ls f''(t) = \frac{\Lambda_{X_1}''(\sqrt{t})-\frac{\Lambda_{X_1}'(\sqrt{t})}{\sqrt{t}}}{2t}
\]
holds for every $t>0$, hence $\Lambda_{X_1}''(s)-\frac{\Lambda_{X_1}'(s)}{s}\gr 0$ for all $s>0$. Now let $x>0$. For $y=\sqrt{x}$ we consider $t_y=(\Lambda^\ast_{X_1})'(y)$. Then $y=\Lambda_{X_1}'(t_y)$, and $\Lambda_{X_1}''(t_y)-\frac{\Lambda_{X_1}'(t_y)}{t_y}\gr 0$ is equivalent to $(\Lambda^\ast_{X_1})''(y)-\frac{(\Lambda^\ast_{X_1})'(y)}{y}\ls 0$ (recall that $(\Lambda^{\ast})''=1/(\Lambda''\circ(\Lambda^{\ast})')$, since $(\Lambda^{\ast})'=\left(\Lambda'\right)^{-1}$). This shows that
\[
\frac{d^2}{dx^2}\Lambda_{X_1}^\ast(\sqrt{x})=\frac{(\Lambda^\ast_{X_1})''(\sqrt{x})-\frac{(\Lambda^\ast_{X_1})'(\sqrt{x})}{\sqrt{x}}}{2x}\ls 0
\]
for the arbitrary $x\in(0,\infty)$, which proves that $x\mapsto \Lambda_{X_1}^\ast(\sqrt{x})$ is concave on $(0,\infty)$.
\end{proof}

Let us record for future use the following direct corollary of Lemma \ref{lem:Barthe-Kold-conv} and Lemma \ref{lem:L-sqrt-convex}.
\begin{corollary}\label{cor:L-sqrt-convex}
Let $X$ be a rotationally invariant random vector in $\mathbb{R}^n$ with density $f_X(x)=f_X(|x|)$. If $f_X\in \mathcal{LC}_{conc}$, then $x(\Lambda_{X_1}^\ast )''(x)-(\Lambda_{X_1}^\ast) '(x)\gr 0$ for every $x>0$. The inequality is reversed if $f_X\in \mathcal{LC}_{conv}$.
\end{corollary} 
\begin{proof}
By Lemma \ref{lem:Barthe-Kold-conv}, $f_X\in \mathcal{LC}_{conc}$ implies that $t\mapsto \Lambda_{X_1}(\sqrt{t})$ is concave, which in turn grants us that $y\mapsto \Lambda_{X_1}^\ast(\sqrt{y})$ is convex on $(0,\infty)$ (Lemma \ref{lem:L-sqrt-convex}). Differentiating twice, it follows that
\[
\frac{1}{4y}\left((\Lambda^{\ast}_{X_1})''(\sqrt{y})-\frac{(\Lambda^{\ast}_{X_1})'(\sqrt{y})}{\sqrt{y}}\right) \gr 0
\]
for every $y>0$, which is equivalent to $x(\Lambda_{X_1}^\ast )''(x)-(\Lambda_{X_1}^\ast) '(x)\gr 0$ for all $x>0$.
\end{proof}

Our primary goal in the rest of this section is the following result.
\begin{theorem}\label{thm:L^ast-maj-gaussian}
If $X$ is an isotropic random vector with a radially symmetric density $f_X$ such that $f_X\in\mathcal{LC}_{conc}$ and $G$ is a standard Gaussian random vector in $\mathbb{R}^n$, then $\Lambda_{X_1}\ls \Lambda_{G_1}$ and, as a consequence, $\Lambda_{X_1}^\ast\gr\Lambda_{G_1}^\ast$. If $f_X\in\mathcal{LC}_{conv}$, the reverse inequalities hold.
\end{theorem}
For the proof, we recall the notion of (second-order) stochastic dominance. If $X,Y : (\Omega,\P)\to \mathbb{R}$ are real random variables in $L_1(\Omega,\P)$, we say that $Y$ stochastically dominates $X$ (or is larger than $X$ in the convex order), and write $X\prec Y$, if
\[
\E X=\E Y \qquad \hbox{ and } \int_x^\infty \P(X>t)\,dt\ls \int_x^\infty \P(Y>t)\,dt \hbox{ for every } x\in\mathbb{R}.
\]
The following Lemma can be found in \cite{Konig-Kwapien} (see also \cite{Mar-Ol}).
\begin{lemma}\label{lem:majorisation}
Suppose that $X,Y : (\Omega,\P)\to \mathbb{R}$ are real random variables in $L_1(\Omega,\P)$.
\vspace{-10pt}
\begin{itemize}
\item[(a)] $X\prec Y$ if and only if $\E\phi(X)\ls \E\phi(Y)$ for every convex function $\phi:\mathbb{R}\to\mathbb{R}_+$ such that $\phi(X)\in L_1(\Omega,\P)$.
\item[(b)] If $X=c$ is a constant random variable equal to $\E Y$, then $c\prec Y$.
\item[(c)] If $\E X=\E Y$ and there are $a,b\in \mathbb{R}$, $a<b$, such that the densities $f_X, f_Y$ satisfy $f_Y-f_X\gr 0$ on $(-\infty,a)\cup(b,\infty)$ and $f_Y-f_X\ls 0$ on $(a,b)$, then $X\prec Y$.
\end{itemize}
\end{lemma}

It is immediate by (b) above that if $X\overset{d}{=}R\vartheta$ is radially symmetric and isotropic then $n\prec R^2$, or equivalently, $|\sqrt{n}\vartheta|^2\prec |X|^2$. It turns out that extra convexity assumptions let us provide Gaussian bounds on $|X|$.
\begin{lemma}\label{lem:stoc-comparison}
Let $X$ be an isotropic log-concave rotationally invariant random vector in $\mathbb{R}^n$. If the density $f_X$ of $X$ satisfies $f_X\in\mathcal{LC}_{conc}$ then $|X|^2\prec |G|^2$. Respectively, if $f_X\in\mathcal{LC}_{conv}$ then $|G|^2\prec |X|^2$.
\end{lemma}
\begin{proof}
If we write $f,g$ for the densities of the random variables $|X|^2,|G|^2$ respectively, we have $f(r)=\frac{1}{2}n\omega_nr^{\frac{n}{2}-1}f_X(\sqrt{r})$ and $g(r)=\frac{1}{2}n\omega_n(2\pi)^{-\frac{n}{2}}r^{\frac{n}{2}-1}e^{-\frac{r}{2}}$. The constraints
\[
\int_0^\infty f(r)-g(r)\,dr=0 \qquad\hbox{ and }\qquad \int_0^\infty r(f(r)-g(r))\,dr=0,
\]
imposed by isotropicity, grant us that $f-g$ has at least two roots $0<x_1<x_2$ on $(0,\infty)$. To verify that $x_1,x_2$ are the only roots on $(0,\infty)$, we examine the sign of $\log f(r)-\log g(r)=\frac{n}{2}\log(2\pi)+\log f_X(\sqrt{r})+\frac{r}{2}$. Recall that the function $r\mapsto\log f_X(\sqrt{r})$ is non-increasing and concave (resp. convex). It follows that it crosses the graph of the affine function $-\frac{n}{2}\log(2\pi)-\frac{r}{2}$ at most twice. Moreover, if $r\mapsto\log f_X(\sqrt{r})$ is concave then the signature of $\log f-\log g$, and therefore $f-g$, must be $-,+,-$. If it is convex, the respective sign changes are $+,-,+$.
\end{proof}
The last ingredient we will need for the proof of Theorem \ref{thm:L^ast-maj-gaussian} is a differential inequality for the moment generating function of the one-dimensional marginals of the spherical distribution.
\begin{lemma}\label{lem:spher-deriv-bounds}
Let $\phi(t)=\E e^{t\vartheta_1}$. Then $\phi'(t)\ls t\phi''(t)$, for every $t>0$.
\end{lemma}
Given Lemma \ref{lem:spher-deriv-bounds}, we can conclude the proof of Theorem \ref{thm:L^ast-maj-gaussian}.
\begin{proof}[Proof of Theorem \ref{thm:L^ast-maj-gaussian}]
For every $t>0$ let $\phi_t(x)=\E e^{t\sqrt{x}\vartheta_1}$. Note that $e^{\Lambda_{X_1}(t)}=\E\phi_t(|X|^2)$ and $e^{\Lambda_{G_1}(t)}=\E\phi_t(|G|^2)$, so in view of Lemma \ref{lem:majorisation}(a) it is enough to verify that $\phi_t$ is convex on $(0,\infty)$. Differentiating twice, we see that
\[
\frac{d^2\phi_t(x)}{dx^2} = \frac{t^2}{4x}\left(\phi''(t\sqrt{x})-\frac{1}{t\sqrt{x}}\phi'(t\sqrt{x})\right),
\]
where $\phi(y)=\E e^{y\vartheta_1}$. The result follows from Lemma \ref{lem:spher-deriv-bounds}.
\end{proof}
We conclude with the proof of Lemma \ref{lem:spher-deriv-bounds}.
\begin{proof}[Proof of Lemma \ref{lem:spher-deriv-bounds}]
Recall that in \eqref{eq:Lambda_theta1} we have verified that $\phi(x)=\Gamma(n/2)h_0(x)$, where for $s=0,1,2,\ldots$ we let
\[
h_s(x) = \left(\frac{x}{2}\right)^{-\frac{n}{2}+1-s}I_{\frac{n}{2}-1+s}(x).
\]
The modified Bessel function satisfies the recurrence relations
\begin{equation}\label{eq:bessel-recurrence}
I_v'(x)=\frac{1}{2}(I_{v-1}(x)+I_{v+1}(x)), \qquad I_{v-1}(x) = \frac{2v}{x}I_v(x)+I_{v+1}(x)
\end{equation}
(see for example \cite[9.6.26]{Abr-Ste}). Differentiating and using \eqref{eq:bessel-recurrence}, we can check that $h_s'(x)=(x/2)h_{s+1}(x)$, for every $s=0,1,2,\ldots$ . In particular,
\begin{align*}
\phi''(x) = \Gamma\left(\frac{n}{2}\right)\left(\frac{1}{2}h_1(x)+\left(\frac{x}{2}\right)^2h_2(x)\right) = \frac{1}{x}\phi'(x)+\Gamma\left(\frac{n}{2}\right)\left(\frac{x}{2}\right)^2h_2(x),
\end{align*}
from which the wanted claim follows (clearly $h_2\gr 0$, cf. the series expansion \eqref{eq:bessel-series}).
\end{proof}

\subsection{Inequalities for the expectation of $\Lambda^\ast_X$}
 
In this section we provide sharp bounds on the expected value of $\Lambda^\ast_X$, and in particular the parameter $\E e^{-\Lambda_X^\ast(X)}$, for rotationally invariant dstributions $X$. We start with the following general upper bound.
\begin{theorem}\label{thm:e^(-L*)-max}
For every rotationally invariant log-concave random vector $X$ in $\mathbb{R}^n$ we have $\E e^{-\Lambda^\ast_X(X)}\ls \E e^{-\Lambda_{\mathcal E}^\ast(\mathcal E)}$.
\end{theorem}

For the proof of Theorem \ref{thm:e^(-L*)-max} we rely on the pointwise comparison established in Proposition \ref{prop:Exp-min} combined with the ``intersecting densities'' argument which has turned out to be a really efficient tool towards the proof of extremal inequalities in the log-concave setting (see for example \cite{ENT}, \cite{CET} and in particular \cite{BNZ}, where a comprehensive exposition of the method in a general framework is included).

\begin{proof}[Proof of Theorem \ref{thm:e^(-L*)-max}]
By the results of the previous section (Proposition \ref{prop:Exp-min}), we have that that $\Lambda^\ast_{X_1}(x)\gr \Lambda^\ast_{\mathcal E_1}(x)$ holds pointwise for any $x>0$, for every isotropic log-concave and spherically symmetric random vector $X$. It follows that for every such $X$, $\E e^{-\Lambda_X^\ast(X)}\ls \E e^{-\Lambda_{\mathcal E}^\ast(X)}=\E e^{-\Lambda_{\mathcal E_1}^\ast(|X|)}$. It remains to show that $\E e^{-\Lambda_{\mathcal E_1}^\ast(|X|)}\ls \E e^{-\Lambda_{\mathcal E_1}^\ast(|{\mathcal E}|)}$.
We denote by $f$ the density of $X$. Recall that the density of $|X|$ is then equal to $n\omega_nr^{n-1}f(r)\mathds{1}_{r>0}$.
\begin{lemma}\label{lem:intersect-dens}
The function $f-f_{\mathcal E}$ changes sign exactly twice on $(0,\infty)$, and its signature is $-,+,-$.
\end{lemma}
\begin{proof}
Since $\log f$ is concave and decreasing on $(0,\infty)$ and $\log f_{\mathcal E}(x)=\log c_n-x\sqrt{n+1}$ is affine, it follows that $\log f-\log f_{\mathcal E}$ changes sign at most twice on $(0,\infty)$. Since $\int_0^\infty f(x)- f_{\mathcal E}(x)\,dx=0$, there must be at least one sign change. If there is only one, say at $x_1>0$, then the constraint $\E |X|^2=\E |{\mathcal E}|^2$ implies that $\int_{0}^\infty(x^2-x_1^2)(f(x)-f_{\mathcal E}(x))\,dx=0$. It is easy to check however that the integrand does not change its sign on $(0,\infty)$, which leads to a contradiction.

Moreover, since $\log f$ and $\log f_{\mathcal E}$ are both decreasing on $(0,\infty)$, if $f(0)>c_n=f_{\mathcal E}(0)$ then the graphs of $\log f, \log f_{\mathcal E}$ would cross at most once on $(0,\infty)$, which we have seen is not the case. This shows that $\log f-\log f_{\mathcal E}$, and hence $f-f_{\mathcal E}$, is negative close to $0$.
\end{proof}

Let us assume that $0<x_1<x_2$ are the two distinct sign changes of $f-f_{\mathcal E}$ on $(0,\infty)$, granted by Lemma \ref{lem:intersect-dens}. Note that, thanks to symmetry and $\E |X|^2=\E |{\mathcal E}|^2$,
\[
\int_{\mathbb{R}^n} e^{-\Lambda_{\mathcal E}^\ast(x)}(f(x)-f_{\mathcal E}(x))\,dx = n\omega_n\int_0^\infty x^{n-1} (e^{-\Lambda_{\mathcal E_1}^\ast(x)}-ax^2-b)(f(x)-f_{\mathcal E}(x))\,dx,
\]
for any $a,b\in\mathbb{R}$. Let $\Psi(x):=e^{-\Lambda_{\mathcal E_1}^\ast(x)}-ax^2-b$. We can choose $a,b$ such that $\Psi(x_1)=\Psi(x_2)=0$. This is because this condition is equivalent to
\[
A\cdot\begin{bmatrix}
a\\ b
\end{bmatrix} = \begin{bmatrix}
e^{-\Lambda_{\mathcal E_1}^\ast(x_1)}\\e^{-\Lambda_{\mathcal E_1}^\ast(x_2)}
\end{bmatrix}
\]
with $A=\begin{bmatrix}
x_1^2 & 1\\x_2^2 & 1
\end{bmatrix}$, so that $\mathrm{det}A=x_1^2-x_2^2<0$. In particular, we can solve explicitly for $a,b$ and get
\[
a=\frac{e^{-\Lambda_{\mathcal E_1}^\ast(x_1)}-e^{-\Lambda_{\mathcal E_1}^\ast(x_2)}}{x_1^2-x_2^2} \qquad\hbox{ and }\qquad b=\frac{x_1^2e^{-\Lambda_{\mathcal E}^\ast(x_2)}-x_2^2e^{-\Lambda_{\mathcal E}^\ast(x_1)}}{x_1^2-x_2^2}.
\]
Notice that $a<0$, due to the monotonicity of $\Lambda^\ast$. Moreover, the following Lemma implies that $a>-1/2$.
\begin{lemma}
The function $x\mapsto e^{-\Lambda^\ast_{\mathcal E_1}(x)}+\frac{x^2}{2}$ is nondecreasing on $(0,\infty)$.
\end{lemma}
\begin{proof}
Letting $\Lambda^\ast=\Lambda^\ast_{\mathcal E_1}$ and $h(x)=e^{-\Lambda^\ast(x)}+x^2/2$, we can see that $h'(x)=x(1-\psi(x))$ for $\psi(x)=\frac{(\Lambda^{\ast})'(x)}{x}e^{-\Lambda^\ast(x)}$. We can check that $\psi$ is decreasing on $(0,\infty)$. This is because
\[
\psi'(x)=-\frac{(\Lambda^{\ast})'(x)^2}{x}e^{-\Lambda^\ast(x)}+\frac{e^{-\Lambda^\ast(x)}}{x^2}((\Lambda^{\ast})''(x)x-(\Lambda^{\ast})'(x))
\]
and $(\Lambda^{\ast})''(x)x-(\Lambda^{\ast})'(x)<0$ (by Corollary \ref{cor:L-sqrt-convex}, since $f_{\mathcal E}\in\mathcal{LC}_{conv}$). Moreover, notice that $\lim_{x\to 0^+}\frac{(\Lambda^\ast)'(x)}{x}$ equals 1 (recall that $\Lambda^\ast$ is given by \eqref{eq:Exp-L*}), which eventually implies that $\psi(x)\ls \lim_{x\to 0^+}\psi(x)=1$ for all $x>0$, and consequently $h'\gr 0$ on $(0,\infty)$.
\end{proof}
The restrictions $a\in(-1/2,0)$ are crucial in proving the following.
\begin{lemma}\label{lem:Psi-monotonicity}
If $-1/2<a<0$, then the function $\Psi'$ changes sign exactly once on $(0,\infty)$, from negative to positive.
\end{lemma}
\begin{proof}
Differentiating, we see that $\Psi'(x)=-x\psi(x)-2ax$, where $\psi(x)=\frac{(\Lambda^{\ast})'(x)}{x}e^{-\Lambda^{\ast}(x)}$ (for $\Lambda^\ast=\Lambda^\ast_{\mathcal E_1}$). We have explained that $\psi$ is decreasing and satisfies $\lim_{x\to 0^+}\psi(x)=1$. It follows that $\Psi'(x)<0$ if $\psi(x)>-2a$, which is true for small enough values of $x$, since $a>-1/2$, and $\Psi'(x)>0$ if $\psi(x)<-2a$.
\end{proof}
We are now in place to complete the proof of Theorem \ref{thm:e^(-L*)-max}: By Lemma \ref{lem:Psi-monotonicity}, we deduce that $\Psi$ changes its monotonicity exactly once on $(0,\infty)$, from decreasing to increasing. This however implies that $x_1,x_2$ are the only roots of $\Psi$ on $(0,\infty)$, which correspond to sign change points. In particular, $\Psi$ must be positive on $(0,x_1)\cup(x_2,\infty)$ and negative on $(x_1,x_2)$. By Lemma \ref{lem:intersect-dens}, it follows then that $\Psi(x)(f(x)-f_{\mathcal E}(x))<0$ everywhere on $(0,\infty)$ and, as a consequence,
\[
\int_{\mathbb{R}^n} e^{-\Lambda_{\mathcal E}^\ast(x)}(f(x)-f_{\mathcal E}(x))\,dx = n\omega_n\int_0^\infty x^{n-1} \Psi(x)(f(x)-f_{\mathcal E}(x))\,dx \ls 0,
\]
which is the desired inequality.
\end{proof}

We can also use Lemma \ref{lem:intersect-dens} and an argument similar to the above, to prove the following.
\begin{theorem}\label{thm:EL*-rotinv-low}
Any rotationally invariant log-concave $X$ in $\mathbb{R}^n$ satisfies $\E\Lambda^\ast_X(X)\gr\E\Lambda^\ast_{\mathcal E}({\mathcal E})$.
\end{theorem}

\begin{proof}
The argument is the same as in the proof of Theorem \ref{thm:e^(-L*)-max}. First we argue that
\[
\E\Lambda_X^\ast(X) = \E\Lambda_{X_1}^\ast(|X|)\gr \E\Lambda_{{\mathcal E}_1}^\ast(|X|), 
\]
since $\Lambda_{X_1}^\ast\gr \Lambda_{\mathcal E_1}^\ast$ pointwise (by Proposition \ref{prop:Exp-min}). Again thanks to the constraint $\E|X|^2=\E|{\mathcal E}|^2$ we now have
\[
\int_{\mathbb{R}^n}\Lambda_{\mathcal E}^\ast(x)(f(x)-f_{\mathcal E}(x))\,dx=n\omega_n\int_0^\infty x^{n-1}(\Lambda_{\mathcal E_1}^\ast(x)-ax^2-b)(f(x)-f_{\mathcal E}(x))\,dx
\]
for any $a,b\in\mathbb{R}$. We now consider the function $\Phi(x)=\Lambda^\ast_{\mathcal E_1}(x)-a x^2-b$, with $a,b$ now chosen so that $\Phi(x_1)=\Phi(x_2)=0$ (where again $0<x_1<x_2$ are the unique sign changes of $f_X-f_{\mathcal E}$ on $(0,\infty)$, granted by Lemma \ref{lem:intersect-dens}). Notice that now $a=\frac{\Lambda^\ast_{\mathcal E_1}(x_1)-\Lambda^\ast_{\mathcal E_1}(x_2)}{x_1^2-x_2^2}\in(0,1/2)$. The lower bound here is obvious, since $x_1<x_2$, and the upper bound follows from the fact that the function $h(x) = \Lambda^\ast_{\mathcal E_1}(x)-x^2/2$ is non-increasing (since $h'(x)=\frac{x\left(1-\sqrt{\frac{4x^2}{n+1}+1}\right)}{1+\sqrt{\frac{4x^2}{n+1}+1}}$ is clearly non-positive).

Like before, we now need to show that $\Phi$ has the same sign as $f_X-f_{\mathcal E}$ on each of the intervals $(0,x_1)$, $(x_1,x_2)$ and $(x_2,\infty)$, so it is enough to verify that $\Phi'$ changes its sign exactly once, from positive to negative. Note that $\Phi'(x)=x(\phi(x)-2a)$, where $\phi(x)=(\Lambda^{\ast}_{\mathcal E_1})'(x)/x$. Since $\phi'(x)=\frac{(\Lambda_{\mathcal E_1}^\ast)''(x)x-(\Lambda_{\mathcal E_1}^\ast)'(x)}{x^2}\ls 0$ (again, by Corollary \ref{cor:L-sqrt-convex}) we get that $\phi$ is non-increasing, with $\phi(x)\to 1$ as $x\to 0$. Since $a<1/2$, it follows then that $\phi(x)>2a$ holds only for $x$ sufficiently close to 0, which shows the desired result for $\Phi'$.
\end{proof}

As a direct corollary of Theorem \ref{thm:e^(-L*)-max}, for the case $n=1$, we get the following bound on the exponential one-shot separability constant of \textit{symmetric} log-concave distributions.
\begin{corollary}\label{cor:one-shot-prod}
Let $X=(X_1,\ldots,X_n)$ be a random vector in $\mathbb{R}^n$ with independent and symmetric log-concave components. Then $\E e^{-\Lambda^\ast_X(X)}\ls (\E e^{-\Lambda^\ast_{Y}(Y)})^n$, where $Y$ is a standard double exponential (note that $\E e^{-\Lambda^\ast_{Y}(Y)}\simeq 0.787...$).
\end{corollary}
\begin{proof}
The case $n=1$ of Theorem \ref{thm:e^(-L*)-max} implies that $\E e^{-\Lambda^\ast_{X_j}(X_j)}\ls \E e^{-\Lambda^\ast_{Y}(Y)}$ for every $j=1,\ldots,n$, and by independence we have that $\E e^{-\Lambda^\ast_X(X)} = \prod_{j=1}^n \E e^{-\Lambda^\ast_{X_j}(X_j)}$.
\end{proof}

Given Theorem \ref{thm:e^(-L*)-max}, one could wonder if a more general result than that of Corollary \ref{cor:one-shot-prod} holds, namely that the right-hand side bound $(\E e^{-\Lambda^\ast_{Y}(Y)})^n$ is true for all rotationally-invariant (not necessarily product) log-concave distributions $X$ in $\mathbb{R}^n$. We already know that for every such $X$, $\E e^{-\Lambda^\ast_X}\ls c^{n}$ for some absolute constant $c\in(0,1)$ (this follows already from \eqref{eq:EL*-BGP} and the well known fact that $L_X$ is bounded in the rotationally invariant case). We believe that the sharp value of $c$ is that provided by Theorem \ref{thm:e^(-L*)-max}. In particular, numerical evidence leads us to the following conjecture. To stress the dependence on the dimension $n$, we denote here by $X_n$ the random vector $\mathcal{E}$ in $\mathbb{R}^n$.
\begin{conjecture}\label{conj:Ee^(-L^ast)-max}
The function
\[
n\mapsto \left(\E e^{-\Lambda^\ast_{X_n}(X_n)}\right)^\frac{1}{n}
\]
is non-increasing.
\end{conjecture}

We close this section by providing improved Gaussian bounds for the parameter $e^{-\Lambda_X^\ast(X)}$ when $f_X$ lies in either of the classes $\mathcal{LC}_{conc}$ or $\mathcal{LC}_{conv}$. The proof of the following theorem employs the same argument as that of Theorem \ref{thm:e^(-L*)-max}.
\begin{theorem}\label{thm:gaussian-Ee^L*}
Let $X$ be a rotationally invariant log-concave random vector in $\mathbb{R}^n$, and assume that $f_X\in\mathcal{LC}_{conv}$. Then $\E e^{-\Lambda_X^\ast(X)}\gr \E e^{-\Lambda_G^\ast(G)}$. If $f_X\in\mathcal{LC}_{conc}$, the inequality is reversed.
\end{theorem}
\begin{proof}
We can assume that $X$ is isotropic. Recall that, by Theorem \ref{thm:L^ast-maj-gaussian}, $e^{-\Lambda^\ast_{X_1}(x)}\gr e^{-\Lambda_{G_1}(x)}=e^{-\frac{x^2}{2}}$ for every $x>0$ if $f_X\in\mathcal{LC}_{conv}$ (and the opposite inequality holds if $f_X\in\mathcal{LC}_{conc}$). It is then enough to prove that $\E e^{-\frac{|X|^2}{2}}\gr \E e^{-\frac{|G|^2}{2}}$ when $f_X\in\mathcal{LC}_{conv}$ (resp. $\E e^{-\frac{|X|^2}{2}}\ls \E e^{-\frac{|G|^2}{2}}$ $f_X\in\mathcal{LC}_{conc}$). As in the proof of Theorem \ref{thm:e^(-L*)-max}, the wanted result will rely on the comparison of the densities $f:=f_{X}$ and $f_{G}$. Note that, after integration in polar coordinates and a change of variables,
\begin{align*}
\E e^{-\frac{|X|^2}{2}}- \E e^{-\frac{|G|^2}{2}} &= n\omega_n\int_0^\infty e^{-\frac{x^2}{2}}x^{n-1}(f(x)-f_{G}(x))\,dx\\
                                              &= \frac{n\omega_n}{2}\int_0^\infty e^{-\frac{y}{2}}y^{\frac{n}{2}}(f(\sqrt{y})-f_G(\sqrt{y}))\,dy.
\end{align*}
As in the proof of Lemma \ref{lem:intersect-dens}, the constraint $\E|X|^2=\E|G|^2$, imposed by isotropicity, will imply that $f(\sqrt{y})-f_G(\sqrt{y})$ has at least two zeroes on $(0,\infty)$. Moreover $\log f(\sqrt{y})$ is nonincreasing and either convex or concave (given that $f_X\in\mathcal{LC}_{conv}$ or $f_X\in\mathcal{LC}_{conc}$, respectively). It follows that it crosses the affine function $\log f_G(\sqrt{y})$ at most twice, hence $f(\sqrt{y})-f_G(\sqrt{y})$ must have exactly two sign changes, say at the points $0<y_1<y_2$, with its sign being $+,-,+$ (if $f_X\in\mathcal{LC}_{conv}$) or $-,+,-$ (if $f_X\in\mathcal{LC}_{conc}$).

Next, consider the function $\Phi(y)=e^{-\frac{y}{2}}-ay-b$, where $a,b$ are chosen so that $\Phi(y_1)=\Phi(y_2)=0$. Namely, we have that
\[
a=\frac{e^{-\frac{y_1}{2}}-e^{-\frac{y_2}{2}}}{y_1-y_2} \qquad\hbox{ and }\qquad b=\frac{y_1e^{-\frac{y_1}{2}}-y_2e^{-\frac{y_2}{2}}}{y_1-y_2}.
\]
Note that the fact that $t\mapsto e^{-t}+t$ is increasing on $(0,\infty)$ implies that $a\in(-1/2,0)$. Differentiating, it is immediate to check that $\Phi$ is convex decreasing for $e^{-y/2}>-2a$ and increasing for $e^{-y/2}<-2a$. It follows that $\Phi$ is first decreasing and then increasing, and since it has at least two zeroes, convexity implies that these are the only two, with its sign being $+,-,+$. Putting everything together, we have verified that $\E e^{-\frac{|X|^2}{2}}- \E e^{-\frac{|G|^2}{2}}\gr 0$ if $f_X\in\mathcal{LC}_{conv}$ (and $\ls 0$ if $f_X\in\mathcal{LC}_{conc}$), which is the wanted statement.
\end{proof}

\begin{remark}
Thanks to independence, it is easy to verify that $\E e^{-\Lambda^\ast_G(G)}=2^{-\frac{n}{2}}$. Theorem \ref{thm:gaussian-Ee^L*} then implies that the upper bound of Corollary \ref{cor:one-shot-prod} can be sharpened to $2^{-\frac{n}{2}}$ for every rotationally invariant log-concave random vector $X$ in $\mathbb{R}^n$ such that $f_X\in\mathcal{LC}_{conc}$ (not necessarily with independent coordinates).
\end{remark}

\noindent The main results of this section are summarised in Table 1 below.

\begin{table}[]
\begin{tabular}{ccccc}
\cline{2-5}
\multicolumn{1}{c|}{}  & \multicolumn{2}{c||}{Extremisers for $\Lambda^\ast_X$}                         & \multicolumn{2}{c|}{Extremisers for $\mathbb{E}e^{-\Lambda^\ast_X(X)}$}                         \\ \cline{2-5} 
\multicolumn{1}{c|}{}  & \multicolumn{1}{c|}{Minimiser} & \multicolumn{1}{c||}{Maximiser} & \multicolumn{1}{c|}{Minimiser} & \multicolumn{1}{c|}{Maximiser} \\ \hline
\multicolumn{1}{|c|}{General $f_X$} & \multicolumn{1}{c|}{${\mathcal E}$} & \multicolumn{1}{c||}{$\sqrt{n}\vartheta$} & \multicolumn{1}{c|}{?} & \multicolumn{1}{c|}{${\mathcal E}$} \\ \hline
\multicolumn{1}{|c|}{$f_X\in\mathcal{LC}_{conc}$} & \multicolumn{1}{c|}{$G$} & \multicolumn{1}{c||}{$\sqrt{n}\vartheta$} & \multicolumn{1}{c|}{?} & \multicolumn{1}{c|}{$G$} \\ \hline
\multicolumn{1}{|c|}{$f_X\in\mathcal{LC}_{conv}$} & \multicolumn{1}{c|}{${\mathcal E}$} & \multicolumn{1}{c||}{$G$} & \multicolumn{1}{c|}{$G$} & \multicolumn{1}{c|}{${\mathcal E}$} \\ \hline
                       &                       &                       &                       &                      
\end{tabular}
\caption{The upper and lower bounds granted by Propositions \ref{prop:L*-general-upper}, \ref{prop:Exp-min} and Theorems \ref{thm:L^ast-maj-gaussian}, \ref{thm:e^(-L*)-max} and \ref{thm:gaussian-Ee^L*}, for a rotationally invariant log-concave random vector $X$ in $\mathbb{R}^n$ with density $f_X$ (for the pointwise bounds on $\Lambda_X^\ast$, $X$ is assumed to be isotropic).}
\label{table:sec4}
\end{table}

\section{Threshold for the expected measure of random polytopes: new results}\label{sec:thresholds}

Let us briefly review the strategy of \cite{BGP} for the expected measure threshold problem. Recall that in general, given $\delta\in(0,1)$ and a probability measure $\mu$ on $\mathbb{R}^n$, we let
\[
\varrho_1(\mu,\delta):= \sup\{r>0: \E\mu(K_N)\ls \delta, \hbox{ for every } N\ls \exp(r)\},
\]
and
\[
\varrho_2(\mu,\delta):= \inf\{r>0: \E\mu(K_N)\gr 1-\delta, \hbox{ for every } N\gr \exp(r)\}.
\]
We also denote
\[
\varrho(\mu,\delta) := \varrho_2(\mu,\delta)-\varrho_1(\mu,\delta).
\] 
Our general goal is to obtain lower and upper bounds for the parameters $\varrho_1(\mu,\delta)$ and $\varrho_2(\mu,\delta)$, respectively. Starting with $\varrho_1$, we apply Lemma \ref{lem:DFM} (a) taking $A=A_\eta=\{x\in\mathbb{R}^n: \Lambda_\mu^\ast(x)\ls (1-\eta)\kappa\}$, $\eta\in(0,1)$, where we denote $\kappa=\E\Lambda_\mu^\ast$. An application of Chebyshev's inequality leads us then to the upper bound
\begin{equation}\label{eq:Chebyshev}
\mu(A)\ls \P\left(|\Lambda^\ast_\mu(X)-\kappa|\gr\eta\kappa\right)\ls \frac{\E|\Lambda^\ast_\mu(X)-\kappa|^2}{\eta^2\kappa^2}.
\end{equation}
This gives rise to the parameter
\begin{equation}\label{eq:beta-defn}
\upbeta(\mu) := \frac{\Var(\Lambda_\mu^\ast)}{(\E\Lambda_\mu^\ast)^2}.
\end{equation}
With \eqref{eq:Chebyshev} as a starting point, the authors in \cite{BGP} also employed \eqref{eq:EL*-BGP} to prove the following result.
\begin{theorem}\label{lower_rho_1}Let $\mu$ be a centered log-concave probability measure on ${\mathbb R}^n$.
Assume that $\upbeta (\mu )<1/8$ and $8\upbeta(\mu )<\delta <1$. If $n/L_{\mu}^2\gr c_2\ln (2/\delta )\sqrt{\delta /\upbeta(\mu)}$
where $L_{\mu}$ is the isotropic constant of $\mu $, then
$$\varrho_1(\mu ,\delta )\gr \left(1-\sqrt{\frac{8\upbeta(\mu)}{\delta}}\right)\mathbb{E}_{\mu}(\Lambda_{\mu}^{*}).$$
\end{theorem}
As a second step, relying on \eqref{eq:L*-lower_BGP}, a matching upper bound for $\varrho_2(\mu,\delta)$ is provided in \cite{BGP}, in the case that the measure is the uniform on a convex body. Using Theorem \ref{thm:L*-q-lower-general}, we can now verify that a similar estimate holds for all log-concave measures.
\begin{theorem}\label{thm:rho2-upper}
Let $\delta>0$ and $\mu$ be any log-concave probability measure on $\mathbb{R}^n$ with $128\upbeta(\mu)<\delta<1$ . Then we can find some $n_0\in\mathbb{N}$ such that
\[
\varrho_2(\mu,\delta)\ls \left(1+\varepsilon(n)\right)\E_\mu(\Lambda^\ast_\mu),
\]
for every $n\gr n_0$, where $c>0$ is an absolute constant, where 
\[
\varepsilon(n)=\max\left\{\sqrt{128\upbeta(\mu)/\delta},\frac{8\log\left(e\left(1+\frac{1}{2}\right)\mathbb{E}\Lambda^{*}\right)}{3\,\mathbb{E}\Lambda^{*}}\right\}.
\]
\end{theorem}

\begin{proof}
We will rely on the lower bound
\[
\E\mu(K_N)\gr\mu(A)\left(1-2\binom{N}{n}\left(1-\inf_{x\in A}q_\mu(x)\right)^{N-n}\right),
\]
of Lemma \ref{lem:DFM}, applied for $A=\{x\in\mathbb{R}^n:\Lambda^\ast\ls\left(1+\frac{\varepsilon}{8}\right)\E\Lambda^\ast\}$, for a proper choice of $\varepsilon\in(0,1)$. On the one hand, by Chebyshev's inequality we have that for any $\varepsilon\in(0,1)$,
\[
\mu\left(\{x\in\mathbb{R}^n:|\Lambda^\ast(x)-\E\Lambda^\ast|\gr\frac{\varepsilon}{8}\E\Lambda^\ast\}\right)\ls \frac{64\upbeta(\mu)}{\varepsilon^2}.
\]
Choosing any $\varepsilon\geq \varepsilon_0:=\sqrt{128\upbeta(\mu)/\delta}$, this implies that $\mu(A)\gr 1-\delta/2$. What we are left with then, is to show that
\begin{equation}\label{eq:rho2-q_lower}
	2\binom{N}{n}\left(1-\inf_{x\in A}q_\mu(x)\right)^{N-n}\ls \frac{\delta}{2}
\end{equation}
for all $N\gr\exp\left((1+\varepsilon)\E\Lambda^\ast\right)$ if $n$ is large. Recall that, by Corollary \ref{cor:q-lower}, we have
\[
\inf_{x\in A}q(x)\gr \exp\left(-\frac{1}{1-\eta}\left(\left(1+\frac{\varepsilon}{8}\right)\E\Lambda^\ast-\log\eta\right)\right),
\]
for every $\eta\in(0,1)$. Set $c_\varepsilon:=(1+\varepsilon/8)\E\Lambda^\ast$ and note that this tends to infinity as $n\to\infty$. Fix $\varepsilon\gr\varepsilon_0$ and choose $$\eta_{\varepsilon}=\left(\left(1+\frac{\varepsilon}{2}\right)\mathbb{E}\Lambda^{*}\right)^{-1}.$$
Then, the inequality  $$\frac{c_\varepsilon-\log\eta}{1-\eta}\ls(1+\varepsilon/2)\E\Lambda^\ast,$$ 
is equivalent to 
$$\varepsilon\gr\frac{8}{3\,\mathbb{E}\Lambda^{*}}\log\left(e\left(1+\frac{\varepsilon}{2}\right)\mathbb{E}\Lambda^{*}\right).$$
Since $\varepsilon_0\leq 1$, it suffices to find $\varepsilon$, such that 
$$\varepsilon\gr\varepsilon_0\ \ \text{and}\ \ \varepsilon\gr\frac{8}{3\,\mathbb{E}\Lambda^{*}}\log\left(e\left(1+\frac{1}{2}\right)\mathbb{E}\Lambda^{*}\right).$$ We can satisfy both inequalities by taking $\varepsilon$, the maximum of the two.  

Finally, to achieve \eqref{eq:rho2-q_lower}, we can assume that $n$ is large enough so that $2^{-n}\ls \delta/2$. Using this assumption together with $\binom{N}{n}\ls(eN/n)^n/e$ and $\log(1+y)\ls y$, we get that
\[
\frac{2}{\delta}\cdot 2\binom{N}{n}\exp\left((N-n)\log(1-\inf_{x\in A}q_\mu(x))\right) \ls \exp\left(n\left(\log(2xe)-(x-1)e^{-(1+\varepsilon/2)\E\Lambda^\ast}\right)\right), 
\]
where $x=N/n$. Eventually, the right hand side is smaller than 1 if $e^{(1+\varepsilon/2)\E\Lambda^\ast}<\frac{x-1}{\log(2xe)}$, which can be achieved if we choose $N\gr\exp((1+\varepsilon)\E\Lambda^\ast)$ (and let $n$ grow even larger, if needed).
\end{proof}

\begin{remark}\label{rem:logEL*/EL*-upper}
Recall that the $\log(\E\Lambda_\mu^\ast)/\E\Lambda_\mu^\ast$ term in the statement of the last Theorem is in general, for any log-concave $\mu$, of the order $O(\log n/n)$. This is due to the general estimates \eqref{eq:EL*-BGP} and \eqref{eq:EL*-upper-general}.
\end{remark}

For any log-concave $\mu$ on $\mathbb{R}^n$ we let $\omega_\mu:=-\log q_\mu$ and consider the parameter
\begin{equation}\label{eq:tau-defn}
\uptau(\mu):=\frac{\Var(\omega_\mu)}{(\E\omega_\mu)^2}.
\end{equation}
Another result of \cite{BGP} for the case of uniform measures on convex sets that we can now extend to the general log-concave setting is the fact that $\upbeta(\mu)$ and $\uptau(\mu)$ behave the same way, when the dimension is large. In particular, $\upbeta(\mu)$ tends to 0 with the dimension if and only if $\uptau(\mu)$ does.
\begin{theorem}\label{thm:beta-tau-comp}
Let $\varepsilon\in(0,1)$. For any log-concave probability measure $\mu$ on $\mathbb{R}^n$,
\begin{multline}\label{eq:beta-tau-comp}
(1-\varepsilon)^2\uptau(\mu)+\frac{2(1-\varepsilon)\log\frac{\varepsilon}{2^{1-\varepsilon}}}{\E\omega}+\varepsilon^2-2\varepsilon\ls \upbeta(\mu)\ls \\ \ls\left(\uptau(\mu)+2\varepsilon-\varepsilon^2-\frac{2(1-\varepsilon)\log\frac{\varepsilon}{2^{1-\varepsilon}}}{\E\omega}\right)\left(\frac{1}{1-\varepsilon}\left(1-\frac{\log\frac{\varepsilon}{2^{1-\varepsilon}}}{\E\Lambda^\ast}\right)\right)^2.
\end{multline}
Consequently,
\begin{equation}\label{eq:b-t-comp}
(1+o(1))\uptau(\mu)-O(L_\mu^2\log n/n)\ls \upbeta(\mu)\ls \left(\uptau(\mu)+O(L_\mu^2\log n/n)\right)(1+O(L_\mu^2\log n/n)).
\end{equation}
Moreover,
\begin{equation}\label{eq:EL*-Eomega-comp}
\left(1-O\left(L_\mu^2\log n/n\right)\right)\E\omega_\mu\ls\E\Lambda^\ast_\mu\ls \E\omega_\mu.
\end{equation}
\end{theorem}
\begin{proof}\phantom{\qedhere}
For the lower bound, by Theorem \ref{thm:L*-q-lower-general} we get that for any $\varepsilon\in(0,1)$,
\[
\E{\Lambda^\ast}^2\gr (1-\varepsilon)^2\E\omega^2+2(1-\varepsilon)\log\frac{\varepsilon}{2^{1-\varepsilon}}\E\omega
\]
and using also $\E\Lambda^\ast\ls \E\omega$ (by \eqref{eq:q-L*-chernoff}) we can write
\begin{align*}
\upbeta(\mu) &= \frac{\Var(\Lambda^\ast)}{(\E\omega)^2}\left(\frac{\E\omega}{\E\Lambda^\ast}\right)^2\gr \frac{(1-\varepsilon)^2\E\omega^2+2(1-\varepsilon)\log\frac{\varepsilon}{2^{1-\varepsilon}}\E\omega-(\E\omega)^2}{(\E\omega)^2}\\ &= (1-\varepsilon)^2\uptau(\mu)+\frac{2(1-\varepsilon)\log\frac{\varepsilon}{2^{1-\varepsilon}}}{\E\omega}+\varepsilon^2-2\varepsilon.
\end{align*}
Since $\E\omega\gr\E\Lambda^\ast\gr cn/L_\mu^2$ we get that for, say $\varepsilon=1/n$,
\[
-\frac{2(1-\varepsilon)\log\frac{\varepsilon}{2^{1-\varepsilon}}}{\E\omega}=O(L_\mu^2\log n/n),
\]
and the left hand side inequality follows.

For the upper bound, we apply $\E\Lambda^\ast\gr(1-\varepsilon)\E\omega+\log\frac{\varepsilon}{2^{1-\varepsilon}}$ and $\E{\Lambda^\ast}^2\ls \E\omega^2$ to get that
\begin{align*}
\upbeta(\mu) &\ls \frac{\E\omega^2-(1-\varepsilon)^2(\E\omega)^2-2(1-\varepsilon)\log\frac{\varepsilon}{2^{1-\varepsilon}}\E\omega}{(\E\omega)^2}\left(\frac{\E\omega}{\E\Lambda^\ast}\right)^2\\
             &\ls\left(\uptau(\mu)+2\varepsilon-\varepsilon^2-\frac{2(1-\varepsilon)\log\frac{\varepsilon}{2^{1-\varepsilon}}}{\E\omega}\right)\left(\frac{1}{1-\varepsilon}\left(1-\frac{\log\frac{\varepsilon}{2^{1-\varepsilon}}}{\E\Lambda^\ast}\right)\right)^2.
\end{align*}
Choosing $\varepsilon=1/n$ we can check that
\[
\left(\frac{1}{1-\varepsilon}\left(1-\frac{\log\frac{\varepsilon}{2^{1-\varepsilon}}}{\E\Lambda^\ast}\right)\right)=O(1+O(L_\mu^2\log n/n))
\]
and
\[
-\frac{2(1-\varepsilon)\log\frac{\varepsilon}{2^{1-\varepsilon}}}{\E\omega} \ls O(L_\mu^2\log n/n)
\]
(recalling again that $\E\omega\gr \E \Lambda^\ast\gr cn/L_\mu^2$). It follows that
\[
\upbeta(\mu)\ls \left(\uptau(\mu)+O(L_\mu^2\log n/n)\right)(1+O(L_\mu^2\log n/n)),
\]
which concludes the proof of \eqref{eq:b-t-comp}. Finally, the lower bound in \eqref{eq:EL*-Eomega-comp} follows again from Theorem \ref{thm:L*-q-lower-general} applied for $\varepsilon=1/n$, since
\begin{equation}\nonumber
\left(1-\frac{1}{n}\right)\E\omega_\mu-\log\frac{n}{2^{1-\frac{1}{n}}}\gr\E\omega_\mu\left(1-\frac{1}{n}-\frac{\log(2n)}{\E\omega_\mu}\right)\gr \E\omega_\mu\left(1-\frac{1}{n}-\frac{L_\mu^2\log(2n)}{cn}\right).\eqno\qed
\end{equation}
\end{proof}

Combining Theorems \ref{lower_rho_1} and \ref{thm:rho2-upper} we obtain the following upper bound for the width $\varrho(\mu,\delta)=\varrho_2(\mu,\delta)-\varrho_1(\mu,\delta)$: There is an absolute constant $C>0$ such that 
\begin{equation}\label{eq:rho-bound}
\varrho(\mu,\delta)\leq C\left(\sqrt{\frac{\upbeta(\mu)}{\delta}}+\varepsilon(n)\right)\E\Lambda^\ast_\mu,
\end{equation}
if $n$ is large enough, where $\varepsilon(n)$ is as in the statement of Theorem \ref{thm:rho2-upper}. This general estimate leads us to new threshold-type results for the asymptotics of the expected measure of random polytopes. We record some consequences in the rest of this section.

\subsection{Product measures: Proof of Theorem \ref{thm:prod-threshold}} \label{sec:thresholds-prod}
For the proof of Theorem \ref{thm:prod-threshold}, we will need uniform upper and lower bounds for the variance and expected value of $\Lambda^\ast_X$ in the one-dimensional log-concave setting. The following is another consequence of Theorem \ref{thm:L*-condition}.
\begin{lemma}\label{lem:var-exp-bounds}
There is an absolute constant $c>0$ such that $\Var(\Lambda_X^\ast)\ls c(\E\Lambda_X^\ast)^2$ for any continuous log-concave random variable $X$.  More precisely, $\Var(\Lambda^\ast_X)<4$ and $\E(\Lambda^\ast_X)>0.1484$.
\end{lemma}
\begin{proof}
We will rely on the fact, stressed again in Remark \ref{rem:q-uniform}, that for a continuous random variable $X$ the random variable $q_X(X)$ is equidistributed to $\min\{U,1-U\}$, where $U$ is uniform on $[0,1]$. Using $q_X\ls e^{-\Lambda^\ast_X}$ we write
\[
\E(\Lambda^\ast_X(X)^2) \ls \E(\log^2(q_X(X))) = 2\int_0^{\frac{1}{2}} \log^2(u)\,du = 2+\log^2(2)+\log 2\simeq 3.866.
\]
To lower bound $\E\Lambda^\ast_X$, we will use the inequality \eqref{eq:L*-lower} for the optimal choice of $\varepsilon$. Differentiating, we can see that the right hand side of \eqref{eq:L*-lower} is maximised when $\varepsilon=(-\log(2q(x)))^{-1}$. Let $A=\{x\in\mathbb{R}: q_X(x)<1/(2e)\}$. For every $x\in A$ then we have that $\varepsilon=(-\log(2q(x)))^{-1}<1$ and so \eqref{eq:L*-lower} gives
\[
\Lambda_X^\ast(x)\gr \log\frac{1}{2q_X(x)\log\frac{1}{2q_X(x)}}-1, \qquad x\in A.
\]
We can now estimate
\begin{align*}
\E\Lambda^\ast_X(X) &\gr \int_A \Lambda^\ast_X(x) f_X(x)\,dx\\
                    &\gr \int_A \left(\log\frac{1}{2q_X(x)\log\frac{1}{2q_X(x)}}-1\right) f_X(x)\,dx = 2\int_0^{\frac{1}{2e}} \left(\log\frac{1}{2u\log\frac{1}{2u}}-1\right)\,du \gr 0.1484, 
\end{align*}
where the equality is due to the distribution of $q_X(X)$, and the Lemma follows.
\end{proof}
We now conclude the Proof of Theorem \ref{thm:prod-threshold}: It is essentially enough to show that $\upbeta(\mu_n)\to 0$ with $n$. Due to independence,
\[
\Var_{\mu_n}(\Lambda_{\mu_n}^\ast) = \Var_{\mu_n}\left(\sum_{j=1}^n \Lambda_{\lambda_j}^\ast\right) = \sum_{j=1}^n \Var_{\lambda_j}(\Lambda_{\lambda_j}^\ast)\ls nM,
\]
where, by Lemma \ref{lem:var-exp-bounds}, $M:=\max_j\Var_{\lambda_j}(\Lambda^\ast_{\lambda_j})<4$. Similarly, $\E\Lambda^\ast_{\mu_n}\gr nL$, where $L:=\min_j\E\Lambda^\ast_{\lambda_j}>0.14$. It follows that for every $n\in\mathbb{N}$,
\[
\upbeta(\mu_n)=\frac{\Var(\Lambda^\ast_{\mu_n})}{(\E\Lambda_{\mu_n}^\ast)^2}\ls Cn^{-1}
\]
for some absolute constant $C>0$. This estimate, together with \eqref{eq:rho-bound}, show that $\varrho(\mu_n,\delta)\ls cn^{\frac{\alpha-1}{2}}$ if $\delta\gr n^{-\alpha}$, $\alpha\in(0,1)$ and $n$ is sufficiently large. In particular, $(\mu_n)$ exhibits a sharp threshold around $\E\Lambda_{\mu_n}^\ast$, in the sense of Definition \ref{def:sharp-threshold}.\qed

We conclude this section with the proof of Proposition \ref{prop:exp-seper-R}, which relies on the argument used in Lemma \ref{lem:var-exp-bounds}.
\begin{proof}[Proof of Proposition \ref{prop:exp-seper-R}]
In the proof of Lemma \ref{lem:var-exp-bounds} we explained that the lower bound
\[
\Lambda_X^\ast(x)\gr \log\frac{1}{2q_X(x)\log\frac{1}{2q_X(x)}}-1
\]
holds for every $x$ in the set $A=\{x\in\mathbb{R}^n:q_X(x)<1/(2e)\}$. We will combine again this with the fact that the random variable $q_X(X)$ is equidistributed to $\min\{U,1-U\}$, where $U$ is uniform on $[0,1]$. Using the trivial bound $e^{-\Lambda_X^\ast(x)}\ls 1$ on $A^c$, we write
\begin{align*}
\E e^{-\Lambda^\ast_X} &= \int_{A}e^{-\Lambda^\ast_X(x)}f_X(x)\,dx + \int_{A^c}e^{-\Lambda^\ast_X(x)}f_X(x)\,dx\\
                       &\ls e\int_A 2q_X(x)\log\frac{1}{2q_X(x)}f_X(x)\,dx+\P(q_X(X)\gr 1/(2e))\\
                       &=e\cdot 2\int_0^{\frac{1}{2e}}2u\log\frac{1}{2u}\,du + \P\left(\frac{1}{2e}\ls U\ls 1-\frac{1}{2e}\right) = \frac{3}{4e}+1-\frac{1}{e},
\end{align*}
which is the promised estimate.
\end{proof}

\subsection{Rotationally invariant $\mathcal{LC}_{conv}$-random polytopes: Proof of Theorem \ref{thm:LC-conv-threshold}}

The proof of Theorem \ref{thm:LC-conv-threshold} will be based on the general estimate \eqref{eq:rho-bound}. To control $\upbeta(X)$, we will rely on the Poincar\'{e} inequality. We say that a random vector $X$ in $\mathbb{R}^n$ satisfies a Poincar\'{e} inuality with constant $\kappa>0$ if for every locally Lipschitz and square integrable $f:\mathbb{R}^n\to\mathbb{R}$,
\begin{equation}\label{eq:poincare}
\Var(f(X))\ls\frac{1}{\kappa}\E|\nabla f(X)|^2.
\end{equation}
We know \cite{Bobkov-spectral} that if $X$ is isotropic and log-concave on $\mathbb{R}^n$ with a radially symmetric distribution, we can take $\kappa^{-1}\ls 13$ in \eqref{eq:poincare}.

We can assume without loss of generality that $X$ is isotropic. Our starting point is \eqref{eq:rho-bound}: There is an absolute constant $C>0$ such that for every $\delta\in(0,1)$ and sufficiently large $n$ we have
\[
\frac{\varrho(X,\delta)}{\E\Lambda^\ast_X}\ls C\sqrt{\frac{\upbeta(X)}{\delta}} \ls C'\sqrt{\frac{\Var(\Lambda_X^\ast)}{\delta n^2}},
\]
since $\E\Lambda_X^\ast\gr cn$ (due to \eqref{eq:EL*-BGP}, since we know that $L_X$ is bounded in the rotational-invariant case, by the main result of \cite{Bobkov-spectral}). To upper bound the latter, we use the assumption $f_X\in\mathcal{LC}_{conv}$. For any $x>0$ let $t=(\Lambda_{X_1}^{\ast})'(x)$. Then we have $0\ls \Lambda_{X_1}^\ast(x)=xt-\Lambda_{X_1}(t)$. By Theorem \ref{thm:L^ast-maj-gaussian} however, $f_X\in\mathcal{LC}_{conv}$ implies that $\Lambda_{X_1}(t)\gr\Lambda_{G_1}(t)=t^2/2$. It follows that $0\ls xt-t^2/2$, equivalently, $(\Lambda_{X_1}^{\ast})'(x)\ls 2x$. Using this bound we get by the Poincar\'{e} inequality that
\begin{align*}
\Var(\Lambda_X^\ast)\ls 13\cdot \E|(\Lambda_{X_1}^{\ast})'(|X|)|^2\ls 52\cdot \E|X|^2 = 52n.
\end{align*}
It follows then that $\lim_{n\to\infty}\frac{\varrho(X,\delta)}{\E\Lambda^\ast_X}=0$ for any $\delta\gr n^{-\eta}$, $\eta\in(0,1)$ ($\varrho(X,\delta)$ is again of the order of $n^{\frac{\eta-1}{2}}$ for this choice of $\delta$).

\begin{remark}
We can verify that if $X$ is such that $f_X\in\mathcal{LC}_{conv}$, then $\E\Lambda^\ast_X$ is of the order of $n$. On the one hand, by Lemma \ref{lem:Barthe-Kold-conv} and Lemma \ref{lem:L-sqrt-convex} it follows that $x\mapsto \Lambda^\ast_{X_1}(\sqrt{x})$ is a concave function. Assuming that $X$ is isotropic, it follows that $\E\Lambda^\ast_{X_1}(R)\ls\Lambda^\ast_{X_1}(\sqrt{n})$ and by Theorem \ref{thm:L^ast-maj-gaussian} the latter is at most $\Lambda_{G_1}^\ast(\sqrt{n})=n/2$. On the other hand, by Theorem \ref{thm:EL*-rotinv-low} we have that $\E\Lambda^\ast_X(X)\gr\E\Lambda^\ast_{\mathcal E}(\mathcal E)$. Recall that in \eqref{eq:Exp-L*} we verified that $\Lambda^\ast_{\mathcal E_1}(x)=\frac{n+1}{2}g(\frac{2x}{\sqrt{n+1}})$, were $g(y)=\sqrt{y^2+1}-1+\log(2(\sqrt{y^2+1}-1))$. By convexity of the Cram\'{e}r transform then,
\[
\E\Lambda^\ast_X\gr \E\Lambda_{\mathcal E_1}^\ast(|{\mathcal E}|)\gr\Lambda_{\mathcal E_1}^\ast(\E|{\mathcal E}|)=\Lambda_{\mathcal E_1}^\ast(\frac{n}{\sqrt{n+1}})=\frac{n+1}{2}g(2n).
\]
Since $g$ is increasing, we eventually get that every rotationally invariant log-concave random vector $X$ in $\mathbb{R}^n$ with $f_X\in\mathcal{LC}_{conv}$ satisfies
\[
(\sqrt{2}-1+\log(2(\sqrt{2}-1)))\cdot\frac{n+1}{2}\ls \E\Lambda^\ast_X\ls \frac{n}{2}.
\]
\end{remark}

\begin{remark}\label{rem:rot-inv}
An example of families of measures covered by Theorem \ref{thm:LC-conv-threshold} is the exponential distribution $\mathcal{E}$ with density proportional to $e^{-|x|}$, considered in Section \ref{sec:rot-inv}, and more generally densities proportional to $e^{-|x|^p}$, $p\in(0,2)$ or Gamma-type densities proportional to $|x|^{\alpha-1}e^{-|x|}$, $0<\alpha<n$.
\end{remark}

\subsection{Beta polytopes: Proof of Theorem \ref{thm:beta-threshold}}
Recall that given $\beta>-1$, $\mu_{n,\beta}$ is the probability distribution on $\mathbb{R}^n$ with density $f_{n,\beta}(x)=c_{n,\beta}(1-|x|^2)^\beta\mathds{1}_{B_2^n}(x)$ on $\mathbb{R}^n$, where $c_{n,\beta}=\frac{\Gamma\left(\beta+n/2+1\right)}{\pi^{\frac{n}{2}}\Gamma(\beta+1)}$. For the proof of Theorem \ref{thm:beta-threshold}, it is again essentially enough to show that $\upbeta(\mu_{n,\beta})\to 0$ as $n\to\infty$. Given Theorem \ref{thm:beta-tau-comp}, it is enough to estimate $\uptau(\mu_{n,\beta})$. We work in a similar fashion  as in \cite{BGP}; by Lemma 2.2 in \cite{BCGTT} we know that $q_{\mu_{n,\beta}}(x)=(1-x^2)^{\beta+\frac{n+1}{2}} h(x,n,\beta)$ for every $x\in(0,1)$, for a function $h$ that satisfies
\begin{equation}\label{eq:h-bounds}
\frac{1}{2\sqrt{\pi}\sqrt{\beta+\frac{n}{2}+1}}\ls h(x,n,\beta)\ls \frac{1}{2x\sqrt{\pi}\sqrt{\beta+\frac{n}{2}}}.
\end{equation}
Using these bounds, calculations of the integrals involved lead us to the following Lemma, whose proof we defer for the end of the section. We use here the standard notation $\psi(x)=\frac{\Gamma'(x)}{\Gamma(x)}$.
\begin{lemma}\label{lem:beta-calc}
For any $n\in\mathbb{N}$ and $-1<\beta\ls c_1n$ for some constant $c_1>0$,
\[
\E\omega_{\mu_{n,\beta}} = \left(\beta+\frac{n+1}{2}\right)\left(\psi\left(\frac{n}{2}+\beta+1\right)-\psi\left(\beta+1\right)\right)+O(\log(\beta+n))
\]
and
\begin{align*}
\E\omega_{\mu_{n,\beta}}^2=\left(\beta+\frac{n+1}{2}\right)^2 & \left(\left(\psi\left(\frac{n}{2}+\beta+1\right)-\psi(\beta+1)\right)^2 +\psi'(\beta+1)-\psi'\left(\frac{n}{2}+\beta+1\right)\right)\\ &\hspace{8cm}+ O(n\log^2(\beta+n)).
\end{align*}
\end{lemma}
Lemma \ref{lem:beta-calc} grants us that
\[
\E \uptau(\mu_{n,\beta}) = O\left(\frac{\psi'(\beta+1)-\psi'\left(\frac{n}{2}+\beta+1\right)}{\left(\psi\left(\frac{n}{2}+\beta+1\right)-\psi(\beta+1)\right)^2}\right).
\]
For simplicity, assume that $n$ is an even integer. Then
\begin{equation}\label{eq:psi-log}
\psi\left(\frac{n}{2}+\beta+1\right)-\psi(\beta+1) = -\sum_{k=1}^{n/2}\frac{1}{\beta+k},
\end{equation}
and since
\[
\sum_{k=1}^{n/2}\frac{1}{\beta+k}\gr \int_1^{\frac{n}{2}+1}\frac{1}{\beta+x}\,dx = \log\left(\frac{n}{2}+\beta+1\right)-\log(\beta+1),
\]
it follows that $\E \uptau(\mu_{n,\beta})\ls c\psi'(\beta+1)\log^{-2}n$ for some absolute constant $c>0$, if $n$ is big enough. The estimate for the order of growth of $\E\Lambda^\ast_{\mu_{n,\beta}}$ follows again immediately by Lemma \ref{lem:beta-calc}, since $\E\Lambda^\ast\sim\E\omega$. The upper bound on $\varrho$ is a consequence of the latter, since the maximum in the statement of Theorem \ref{thm:rho2-upper} is then attained by the $\sqrt{\upbeta(\mu_{n,\beta})/\delta}$ term, combined with Theorem \ref{thm:beta-tau-comp} and our upper estimate of $\E\uptau(\mu_{n,\beta})$.

It remains to prove Lemma \ref{lem:beta-calc}.
\begin{proof}[Proof of Lemma \ref{lem:beta-calc}]
Since
\[
B(x,y)=\int_0^1 s^{x-1}(1-s)^{y-1}\,ds=\frac{\Gamma(x)\Gamma(y)}{\Gamma(x+y)},
\]
we can see that for any $k,l=0,1,2,\ldots$,
\begin{equation}\label{eq:beta-deriv}
\int_0^1 s^{x-1}(1-s)^{y-1}\log^k(s)\log^l(1-s)\,ds = \frac{\partial^k}{\partial x^k}\frac{\partial^l}{\partial y^l}B(x,y).
\end{equation}
This lets us compute the integrals involved in $\E\omega_{\mu_{n,\beta}}$ and $\E\omega^2_{\mu_{n,\beta}}$ in terms of the derivatives of the Beta function. We start with
\[
\E\omega_{\mu_{n,\beta}} = b_{n,\beta}\int_0^1 s^{\frac{n}{2}-1}(1-s)^\beta \log(1/q_{\mu_{n,\beta}}(\sqrt{s}))\,ds,
\]
where here and for the rest of the proof we denote $b_{n,\beta}=(B(n/2,\beta+1))^{-1}$. Recall that
\[
-\log q(\sqrt{s}) = -\left(\beta+\frac{n}{2}+1\right)\log(1-s)-\log h(\sqrt{s},n,\beta),
\]
where $h$ satisfies \eqref{eq:h-bounds}. We compute
\begin{equation}\label{eq:int-log(1-s)}
\int_0^1 s^{\frac{n}{2}-1}(1-s)^\beta \log(1-s)\,ds = \frac{\partial B(n/2,y)}{\partial y}\Big|_{y=\beta+1} = b_{n,\beta}^{-1}(\psi(\beta+1)-\psi(n/2+\beta+1)),
\end{equation}
and
\begin{equation}\label{eq:int-logs}
\int_0^1 s^{\frac{n}{2}-1}(1-s)^\beta \log(s)\,ds = \frac{\partial B(x,\beta+1)}{\partial x}\Big|_{x=\frac{n}{2}} = b_{n,\beta}^{-1}(\psi(n/2)-\psi(n/2+\beta+1)).
\end{equation}
By the latter and \eqref{eq:h-bounds}, we can check that the term $b_{n,\beta}\int_0^1 s^{\frac{n}{2}-1}(1-s)^\beta(-\log h(\sqrt{s},n,\beta))\,ds$ is lower bounded by $\log(2\sqrt{\pi}\sqrt{\beta+n/2})+\frac{1}{2}(\psi(n/2)-\psi(n/2+\beta+1))$ and upper bounded by $\log(2\sqrt{\pi}\sqrt{n/2+\beta+1})$. The assumption $\beta\ls c_1n$ grants us that $\psi(n/2)-\psi(n/2+\beta+1)\to 0$ with $n$, and putting things together we eventually get to
\[
\E\omega_{\mu_{n,\beta}} =\left(\beta+\frac{n}{2}+1\right)\left(\psi\left(\frac{n}{2}+\beta+1\right)-\psi(\beta+1)\right)+O(\log(\beta+n/2)).
\]
For the second moment, we will have two consider integration of the terms $\log^2(1-s)$, $\log(1-s)\log h(\sqrt{s},n,\beta)$ and $\log^2h(\sqrt{s},n,\beta)$ against the Beta density. We first use the $(k=0,l=2)$-case of \eqref{eq:beta-deriv} to arrive at
\[
b_{n,\beta}\int_0^1 s^{\frac{n}{2}-1}(1-s)^\beta\log^2(1-s)\,ds = (\psi(\beta+1)-\psi(n/2+\beta+1))^2+\psi'(\beta+1)-\psi'(n/2+\beta+1).
\]
For the second term we combine the upper bound $h(\sqrt{s},n,\beta)^{-1}\ls 2\sqrt{\pi}\sqrt{\beta+n/2}$ with the computation \eqref{eq:int-log(1-s)} to get
\[
b_{n,\beta}\int_0^1 s^{\frac{n}{2}-1}(1-s)^\beta\log(1-s)\log h\,ds \ls \log(2\sqrt{\pi}\sqrt{\beta+n/2})(\psi(n/2+\beta+1)-\psi(\beta+1)).
\]
Finally, to upper bound the $\log^2h$ term we check with the aid of \eqref{eq:h-bounds} that
\[
\log^2h(\sqrt{s},n,\beta)\ls \log^2(2\sqrt{\pi}\sqrt{n/2+\beta+1})+\log^2(2\sqrt{\pi s}\sqrt{n/2+\beta}),
\]
and using the $(k=2,l=0)$-case of \eqref{eq:beta-deriv} we can verify that
\begin{align*}
&b_{n,\beta}\int_0^1 s^{\frac{n}{2}-1}(1-s)^\beta\log^2h\,ds \ls\\ &\hspace{3cm}\ls\frac{1}{4}\left((\psi(n/2)-\psi(n/2+\beta+1))^2+\psi'(n/2)-\psi'(n/2+\beta+1)\right)+c\log n.
\end{align*}
Putting everything together, and checking that the estimates for the last two terms are at most $O(\log^2(\beta+n/2))$ (recall \eqref{eq:psi-log}), we eventually get that indeed
\[
\E\omega_{\mu_{n,\beta}}^2=\left(\beta+\frac{n+1}{2}\right)^2\left(\left(\psi\left(\frac{n}{2}+\beta+1\right)-\psi(\beta+1)\right)^2 +\psi'(\beta+1)-\psi'\left(\frac{n}{2}+\beta+1\right)\right),
\]
modulo a $O((\beta+n/2)\log^2(\beta+n/2))$-term.
\end{proof}

\textbf{Acknowledgement.} We would like to thank the anonymous referee for a number of comments that helped us improve the exposition. The research project is implemented in the framework of H.F.R.I call “Basic research Financing (Horizontal support of all Sciences)” under the National Recovery and Resilience Plan “Greece 2.0” funded by the European Union –NextGenerationEU(H.F.R.I. Project Number:15445).

\bibliographystyle{abbrv}

\bibliography{biblio}

\end{document}